\documentclass[11pt,letterpaper]{amsart}%
\usepackage{amsmath}
\usepackage{txfonts}
\usepackage{graphicx}
\usepackage{amsthm}
\usepackage{amsfonts}
\usepackage{amssymb}
\usepackage[utf8]{inputenc}
\usepackage{comment}
\usepackage{xcolor}
\usepackage{lipsum}
\usepackage{ifthen}%
\providecommand{\U}[1]{\protect\rule{.1in}{.1in}}
\newboolean{bluecomments}
\setboolean{bluecomments}{false}
\newcommand{\bluecomment}[1]{\ifthenelse{\boolean{bluecomments}}{\textcolor{blue}{#1}}{}}
\newboolean{redcomments}
\setboolean{redcomments}{true}
\newcommand{\redcomment}[1]{\ifthenelse{\boolean{redcomments}}{\textcolor{red}{#1}}{}}
\newcounter{mycounter}
\numberwithin{mycounter}{section}
\newtheorem{theorem}[mycounter]{Theorem}
\theoremstyle{plain}

\newtheorem{claim}[mycounter]{Claim}

\newtheorem{lemma}[mycounter]{Lemma}

\newtheorem{proposition}[mycounter]{Proposition}

\numberwithin{equation}{section}
\theoremstyle{definition}

\newtheorem{remark}[mycounter]{Remark}

\begin{document}
\title[A Geometric flow towards hamiltonian stationary submanifolds]{A Geometric flow towards hamiltonian stationary submanifolds}
\author{Jingyi Chen and Micah Warren}
\address{Department of Mathematics\\
The University of British Columbia, Vancouver, BC V6T 1Z2}
\email{jychen@math.ubc.ca}
\address{
Department of Mathematics\\
University of Oregon, Eugene, OR 97403}
\email{micahw@uoregon.edu}
\thanks{Chen is partially supported by NSERC Discovery Grant GR010074}
\maketitle

\begin{abstract}
In this paper, we consider a geometric flow for Lagrangian submanifolds in an almost
K\"ahler manifold that stays in its initial Hamiltonian isotopy class and is a
gradient flow for volume. The stationary solutions are the Hamiltonian
stationary Lagrangian submanifolds. The flow is not strictly parabolic but it
corresponds to a fourth order strictly parabolic scalar equation in the
cotangent bundle of the submanifold via Weinstein's Lagrangian neighborhood
theorem. For any compact initial Lagrangian immersion, we establish short-time
existence, uniqueness, and higher order estimates when the second fundamental
forms are uniformly bounded up to time $T$.

\end{abstract}

\section{Introduction}

The objective of this paper is to study a fourth order flow of Lagrangian
submanifolds in an almost K\"ahler manifold as a gradient flow of volume within a
Hamiltonian isotopy class and establish basic properties such as short-time
existence, uniqueness, and extendibility with bounded second fundamental form.

\vspace{0.1cm}Our setting includes an almost K\"{a}hler manifold $(M^{2n},h,\omega,J)$
with symplectic form $\omega$ and compatible {Riemannian} metric $h$ satisfying
$h(JV,W)=\omega(V,W)$ where $J$ is an almost complex structure, and a given compact
Lagrangian immersion $\iota:L^{n}\rightarrow M^{2n}.$ We propose to find
$F:L\times\lbrack0,T)\rightarrow M^{2n}$ satisfying
\begin{align}
&  \frac{dF}{dt}=J\nabla\operatorname{div}\left(  JH\right) \label{flow0}\\
&  F\left(  \cdot,0\right)  =\iota\left(  \cdot\right)  \label{flow0b}%
\end{align}
where $H$ is the mean curvature vector of $L_{t}=F(\cdot,t)$ in $M$ and
$\nabla,\operatorname{div}$ are along $L_{t}$ in the induced metric from $h$.

The stationary solutions of the above evolution equation are the so-called
Hamiltonian stationary Lagrangian submanifolds, which are fourth order
generalizations of the special Lagrangians, and exist in more abundance. Their
significance lies in the fact that they are critical points of the volume
functional among Lagrangian submanifolds in the same Hamiltonian isotopy
class, in particular include minimal Lagrangian ones which are of great
geometric interest with applications in mirror symmetry. There are independent
reasons to be interested in Hamiltonian stationary Lagrangian submanifolds:
When the Hamiltonian class is zero-Maslov, it is possible for the Hamiltonian
stationary submanifold to be absolutely volume minimizing, hence special
Lagrangian in a Calabi-Yau manifold.
On the other hand, compact Hamiltonian stationary submanifolds with
non-trivial Maslov class (e.g. Clifford torus) are interesting geometric
objects on their own right, and are studied in relation to the Willmore
Conjecture in dimension 4 \cite{Minicozzi} and Oh's Conjecture \cite{Oh1993}.

It is interesting to compare the extensively studied mean curvature flow (MCF)
with our new flow \eqref{flow0}. First, they are both (negative $L^{2}$)
gradient flows for volume, the difference is that the former is along smooth
vector fields and latter is along Hamiltonian vector fields. Both flows share
the same extension property: if the second fundamental form is bounded up to a
time $T<\infty$ then the flow extends with estimates of any order. Second, in
light of the fact that there are no compact minimal submanifolds in
$\mathbb{C}^{n}$
and that it is not always possible to find minimal Lagrangians even in the
integral Lagrangian homology class \cite{MicallefWolfson}, it is feasible that
the flow \eqref{flow0} might be more robust (say avoiding singularities) than
MCF.

Recall that two manifolds are \textit{Hamiltonian isotopic} if they can be
joined by
the flow generated by a vector field of the form $J\nabla f$ for a function
$f$.

\vspace{.1cm}

We will show that the initial value problem \eqref{flow0} - \eqref{flow0b}:

\begin{enumerate}
\item stays within the Hamiltonian isotopy class - that is, the flow is
generated by a vector field of the form $J\nabla f,$

\item is the gradient flow of volume with respect to an appropriate metric, so
decreases volume along the flow,

\item enjoys short time existence given smooth initial conditions, and

\item continues to exist as long as a second fundamental form bound is satisfied.
\end{enumerate}

The flow \eqref{flow0} is degenerate parabolic but not strictly parabolic. Our
proof of existence and uniqueness involves constructing global solutions
(locally in time) using Weinstein's Lagrangian neighbourhood theorem, which
results in a nice fourth order parabolic \textit{scalar} equation. This
equation has a good structure, satisfying conditions required in \cite{MM} so
we can conclude uniqueness and existence within a given Lagrangian
neighbourhood. We then argue that the flows described by solutions to the
scalar fourth order equations are in correspondence with the normal flows of
the form (\ref{flow0}) leading to uniqueness and extendability provided the
flow remains smooth.

After proving existence and uniqueness using Weinstein neighborhoods, we turn
to local Darboux charts to prove higher regularity from second fundamental
form bounds. It's not immediately clear how to extract regularity from
arbitrary Weinstein neighbourhoods as the submanifolds move, but using a
description of Darboux charts in \cite{JLS} we can fix a finite set of charts
and perform the regularity theory in these charts. Unlike mean curvature flow,
there is no maximum principle for fourth order equations. We deliver a
regularity theory using Sobolev spaces. Smoczyk \cite{Smoczyk99} has shown the
basic fact that mean curvature flow preserves the Lagrangian property in a K\"ahler-Einstein manifold, by a
maximum principle argument, while our flow achieves this by evolving within a
Hamiltonian class  in an almost K\"ahler manifold.


\vspace{.1cm}

The newly introduced flow already exhibits nice features in special cases:

\noindent(1) For Calabi-Yau manifolds, Harvey and Lawson \cite{HarveyLawson82}
showed that, for a Lagrangian submanifold, the Lagrangian angle $\theta$
generates the mean curvature via the striking relation
\[
H=J\nabla\theta.
\]
In this case (\ref{flow0}) becomes
\[
\frac{dF}{dt}=-J\nabla\Delta\theta, \label{flow}
\]
while $\theta$ satisfies the pleasant fourth order parabolic equation
\[
\frac{d\theta}{dt}=-\Delta_{g}^{2}\theta
\]
where $g$ is the induced metric on $L_{t}$.

\medskip

\noindent(2) The case $n=1$ involves curves in $\mathbb{C}$. Hamiltonian
isotopy classes bound a common signed area. Higher order curvature flows have
been studied and are called polyharmonic heat flow or curve diffusion flow,
these deal with evolution in the form $\gamma^{\prime}=(-1)^{ p-1}%
\kappa_{s^{2p}}$ where $\kappa_{s^{2p}}$ is the $2p$-th order derivative of
the curvature $\kappa$ of a plane curve $\gamma$ with respect to the arclength
$s$. For $p=0$ it is the standard curve shortening flow. The flow discussed
here corresponds to $p=1$, namely the curve diffusion flow. It appears this
flow was first described in \cite[Chapter 3]{Polden} as the ``obvious
fourth-order analogue" to curve shortening flow, where a gradient nature of
the flow is suggested and a simple argument was given that immersed figure 8
curves must have finite time singularities. In \cite{Fife}, the gradient flow
structure was explored in relation to Cahn-Hilliard equations. In subsequent
works \cite{DKS}, \cite{Chou}, \cite{EscherIto}, \cite{Wheeler},
\cite{ParkinsWheeler} and others, time-to-singularity estimates were given, as
well as conditions for stability when the initial curve is near the circle.
Unlike the curve shortening flow, there are self-similar figure 8 curves which
can be explicitly described \cite{figure8}. It remains an open question
whether an embedded closed curve with winding number $1$ can form a
singularity. It is known that loss of convexity, embeddedness, and
graphicality \cite{GigaIto99}, \cite{GigaIto98}, \cite{graphicality} can all occur.

\begin{remark} Recently, we became aware of the paper \cite{Sch2013} by Lars Sch\"afer, in which the flow  \eqref{flow0} is introduced, and an argument for short-time existence and uniqueness for the K\"ahler-Einstein setting is given.  
\end{remark}

\section{Gradient Flow}

We begin by setting up the equation as a formal gradient flow over an $L^{2}$
metric space. Given an embedded Lagrangian submanifold $L^{n}\subset M^{2n}$,
we can consider Hamiltonian deformations of $L$, these will be flows of vector
fields $J\nabla f$ for scalar functions $f$ on $M$. At $x\in L$, the normal
component of $J\nabla f$ is given by $J\nabla_{L}f$ where $\nabla_{L}f$ is the
gradient of $f$ as a function restricted to $L$; conversely, given any smooth
function $f$ on an embedded $L,$ we can extend $f$ to a function on $M$ so
that $\nabla_{L}f$ is not changed and is independent of the extension of $f$
(see section \ref{related def} below).

In other words, given a family of $C^{1}$ functions $f\left(  \cdot,t\right)
$ along $L,$ one can construct a family of embeddings
\begin{equation}
F:L\times\left(  -\varepsilon,\varepsilon\right)  \rightarrow M \label{path1}%
\end{equation}
satisfying
\begin{equation}
\frac{d}{dt}F(x,t)=J\nabla f\left(  x,t\right)  \label{deform}%
\end{equation}
and conversely, given any path (\ref{path1}) within a Hamiltonian isotopy
class, there will be a function $f$ so that (possibly after a diffeomorphism
to ensure the deformation vector is normal) the condition (\ref{deform})\ is satisfied.

Let $\mathcal{I}_{L_{0}}$ be the set of smooth manifolds that are Hamiltonian
isotopic to $L_{0}$. The (smooth) tangent space at any $L\in\mathcal{I}%
_{L_{0}}$ is parameterized via
\begin{equation}
T_{L}\mathcal{I}_{L_{0}}\mathcal{=}\left\{  f\in C^{\infty}(L): \int%
_{L}fdV_{g}=0\right\}  . \label{tangentspaceH}%
\end{equation}
We use the $L^{2}$ metric on $T_{L}\mathcal{I}_{L_{0}}$: For $JX_{i}=\nabla
f_{i}$ $\in T_{L}\mathcal{I}_{L_{0}}, i=1,2$
\begin{equation}
\langle X_{1},X_{2}\rangle=\int_{L}f_{1}f_{2}\,dV_{g}. \label{h1met}%
\end{equation}
With volume function given by
\[
\operatorname{Vol}(L)=\int_{L}dV_{g}
\]
the classical first variation formula gives
\[
d\operatorname{Vol}(L)(W)=-\int_{L}\langle W,H\rangle dV_{g}
\]
where $W$ is the deformation vector field and $H$ is the mean curvature
vector. In the situation where allowable deformation vectors are of the form
$W=J\nabla f,$ we get
\begin{align*}
d\operatorname{Vol}(L)(f)  &  =-\int_{L}\langle J\nabla f,H\rangle dV_{g}\\
&  =\int_{L}\langle\nabla f,JH\rangle dV_{g}\\
&  =-\int_{L}f\operatorname{div}\left(  JH\right)  dV_{g}%
\end{align*}
using the fact that $J$ is orthogonal, then integrating by parts. Note that
$-\operatorname{div}\left(  JH\right)  $ belongs to $T_{L}\mathcal{I}_{L_{0}}$
since it integrates to 0 on $L$. Therefore, it is the gradient of the volume
function with respect to the metric. Thus a (volume decreasing) gradient flow
for volume would be a path satisfying
\begin{equation}
\frac{dF}{dt}=J\nabla\operatorname{div}\left(  JH\right)  . \label{flow1}%
\end{equation}

\begin{remark}
The metric \eqref{tangentspaceH} is not the usual $L^{2}$ metric for
deformations of a submanifold, which would measure the length of the tangent
vector by $\int\left\vert J\nabla f\right\vert ^{2}dV_{g}. $
It is better suited than the standard metric on vector fields. Suppose instead
we take the \textquotedblleft standard" $L^{2}$ metric on deformation fields:%
\[
d\operatorname{Vol}(L)(J\nabla f)=-\int_{L}\langle J\nabla f,H\rangle dV_{g}%
\]
The gradient with respect to this metric would be $J\nabla\eta$ for some
$\eta\in C^{\infty}(L)$ such that
\[
-\int_{L}\langle J\nabla f,H\rangle=\int_{L}\langle J\nabla f,J\nabla
\eta\rangle.
\]
While the Lagrangian angle $\theta$ (in the Calabi-Yau case) does produce this
gradient locally, typically $\theta$ is not globally defined on $L$. So
instead, we must find a function $\eta$ solving the equation
\[
\Delta\eta=-\operatorname{div}(JH)
\]
which one can solve uniquely up to additive constants since $L$ is compact and
$\nabla\eta= -JH$ + $X$ (divergence free vector field on $L$). A gradient flow
would be $dF/dt = H -JX$ but there is no canonical way to determine $X$.

It is also worth noting that gradient flow with respect to the $L^{2}$ metric
(sometimes called $\mathcal{H}^{-1}$) is not new: It has been used for example
in mechanics to describe the flow of curves \cite{Fife} and was alluded to in
Polden's thesis \cite{Polden}. It bears similarity to a metric used to study
almost calibrated Lagrangian submanifolds (up to a factor of $\cos\theta$ on
the volume form) introducted by Solomon \cite{Solomon1}.
\end{remark}


\subsection{Related definitions of Hamiltonian deformations\label{related def}%
}

Traditionally, Hamiltonian isotopies are defined as flows of the entire
manifold along the direction of a time-dependent vector field $J\nabla f$ for
some $f$ a smooth function on $M$. Two submanifolds are Hamiltonian isotopic
if the one submanifold is transported to the other via the isotopy. In order
to use the description (\ref{tangentspaceH}), we note the following standard result:

\begin{lemma}
\label{both}For a smooth flow of embedded Lagrangian submanifolds satisfying
\begin{equation}
\frac{dF}{dt}=J\nabla f(\cdot,t)\text{ } \label{subh}%
\end{equation}
for some function $f$ defined on $L$ for each $t$, there exists a function
$\tilde{f}$ on $M\times[0,1]$ that defines a Hamiltonian isotopy on $M$ and
determines the same Hamiltonian isotopy of the submanifolds.

Conversely, given a global Hamiltonian isotopy determined by $\tilde{f}$, the
function $\tilde{f}$ restricted to $L_{t}$ determines a flow of the form
(\ref{subh}), possibly up to reparameterization by diffeomorphisms of $L.$
\end{lemma}

\begin{proof}
The function $f$ is defined on a smooth compact submanifold of $M\times
\lbrack0,1]$. We can use the Whitney extension theorem to extend a smooth
function off this set in which the normal derivatives vanish. Thus along the
Lagrangian submanifolds, $J\nabla f=J\nabla\tilde{f}$.

Conversely, given any function $\tilde{f}$ its gradient decomposes into the
normal and tangential parts on the Lagrangian submanifold. By the Lagrangian
condition, $J\nabla^{T}\tilde{f}$ is normal and $J\nabla^{\perp}\tilde{f}$ is
tangential with the latter component describing merely a reparameterization of
$L$. So the flow is completely determined by the component $J\nabla^{T}%
\tilde{f}$ which is determined by the restriction to $L$.
\end{proof}

\subsection{Immersed Lagrangian submanifolds and their Hamiltonian
deformations}

\label{section}

Along the evolution equation (\ref{flow0}), it is feasible that a submanifold
which is initially embedded will become merely immersed. Thus we would like
the equation to behave well even when the submanifold is immersed.

Weinstein's Lagrangian neighborhood Theorem for immersed Lagrangian
submanifolds \cite[Theorem 9.3.2]{EM} states that any Lagrangian immersion
$F_{0}: L\to M$ extends to an immersion $\Psi$ from a neighborhood of the
0-section in $T^{*}L$ to $M$ with $\Psi^{*}\omega_{M}=\omega_{\text{can}}$.

Sections of the cotangent bundle $T^{*}L$ are clearly embedded as graphs over
the 0-section of $T^{*}L$ which is identified with $L$, so by factoring the
immersion through $T^{*}L$, we get immersed submanifolds in $M$, in
particular, immersed Lagrangian submanifolds in $M$ for sections defined by
closed 1-forms on $L$.

Even though the deformation of an immersed manifold is not properly
Hamiltonian (that is, velocity vector $J\nabla f$ determined by a global
function $f$ on $M$) one can define deformations by using a function $f$
defined on the submanifold, and $J\nabla f$ makes sense within $T^{*}L$ as
pullback by the immersion. For example, the figure 8 is not problematic
because the two components of a neighborhood of the crossing point can have
different velocity vectors; these are separated within the cotangent bundle.


\subsection{The evolution equation in terms of $\theta$}

By \cite{HarveyLawson82}, for a Lagrangian submanifold $L$ in a Calabi-Yau
manifold $(M^{n}, \omega, J, \Omega)$ with a covariant constant holomorphic
$n$-form $\Omega$, the mean curvature of $L$ satisfies $H=J\nabla\theta$
where
\[
\Omega|_{L}=e^{i\theta}d\operatorname{Vol}_{L}.
\]
Now \eqref{flow1} leads to
\[
\frac{dF}{dt}=J\nabla\operatorname{div}\left(  JH\right)  =J\nabla
\operatorname{div}\left(  JJ\nabla\theta\right)  =-J\nabla\Delta\theta.
\]
Differentiating the left-hand side (cf. \cite[Prop 3.2.1]{WoodThesis}):%

\begin{align*}
\frac{d}{dt}\Omega|_{L}  &  =\frac{d}{dt}\left(  F_{t}^{\ast}\Omega\right)
=F_{0}^{\ast}\mathcal{L}_{-J\nabla\Delta\theta}\Omega\\
&  =F_{0}^{\ast}d\left(  \iota_{-J\nabla\Delta\theta}\Omega\right)  =d\left(
F_{0}^{\ast}\left(  \iota_{-J\nabla\Delta\theta}\Omega\right)  \right) \\
&  =d\left(  F_{0}^{\ast}i\left(  \iota_{-\nabla\Delta\theta}\Omega\right)
\right) \\
&  =d\left(  ie^{i\theta}dVol_{L}(-\nabla\Delta\theta,\cdot,...,\cdot)\right)
\\
&  =-d\left(  ie^{i\theta}\ast d\Delta\theta\right) \\
&  =\left(  e^{i\theta}\ast d\Delta\theta-ie^{i\theta}d\left(  \ast
d\Delta\theta\right)  \right)  .
\end{align*}
Then differentiating the right hand side:
\[
\frac{d}{dt}e^{i\theta}d\operatorname{Vol}_{L}=e^{i\theta}\frac{d}%
{dt}d\operatorname{Vol}_{L}+ie^{i\theta}\frac{d\theta}{dt}d\operatorname{Vol}%
_{L}.
\]
Comparing the imaginary parts (after multiplying by $e^{-i\theta}$) of the
above two gives
\begin{equation}
\frac{d\theta}{dt}=-\Delta_{g(t)}^{2}\theta. \label{bilaplace}%
\end{equation}



\section{Existence and Uniqueness via a scalar equation on a Lagrangian
Neighborhood\label{LagN}}

The system of equations (\ref{flow0}) is not strictly parabolic as given. Our
approach is to make good use of the Lagrangian property, in particular, by
setting up the equation as a scalar, uniformly parabolic equation via
Weinstein's Lagrangian neighborhood theorem. For the convenience of using
common terminologies, we make our discussion for embeddings but the
conclusions hold for immersions in view of subsection \ref{section}.

\subsection{Accompanying flow of scalar functions\label{existence}}

Let $L$ be an embedded compact Lagrangian submanifold in a symplectic manifold
$(M,\omega)$. By Weinstein's Lagrangian neighborhood \cite[Corollary
6.2]{Weinstein} theorem, there is a diffeomorphism $\Psi$ from a neighborhood
$U\subset T^{*}L$ of the 0-section (identified with $L$) to a neighborhood
$V\subset M$ of $L$ such that $\Psi^{\ast}\omega=d\lambda_{can}$ and $\Psi$
restricts to the identity map on $L$.

Let $\varphi(x,t)$ be a smooth function on $L\times\lbrack0,\delta)$ with
$\varphi(\cdot,0)=0$. Then $d\varphi$ is a $t$-family of exact (hence closed)
1-forms on $L$ hence a family of sections of $T^{\ast}L$ and each is a graph
over the 0-section. The symplectomorphism $\Psi$ yields a $t$-family of
Lagrangian submanifolds $L_{t}$ in $M$ near $L$:
\begin{equation}
F=\Psi(x,d\varphi(x,t))=\Psi\left(  x,\frac{\partial\varphi}{\partial x^{k}%
}dx^{k}\right)  . \label{F}%
\end{equation}

\begin{proposition}
\label{flow equation}Suppose that $d\varphi(x,t)$ is an exact section
describing an evolution of Lagrangian submanifolds which satisfy the equation
\begin{equation}
\left(  \frac{dF}{dt}\right)  ^{\perp}=J\nabla\operatorname{div}\left(
JH\right)  . \label{perpMM}%
\end{equation}
Then there is a function $G$ (depending on $\Psi$) such that $\varphi$
satisfies
\[
\frac{\partial\varphi}{\partial t}=-g^{ap}g^{ij}\frac{\partial^{4}\varphi
}{\partial x^{a}\partial x^{j}\partial x^{i}\partial x^{p}}+G(x,D\varphi
,D^{2}\varphi,D^{3}\varphi).
\]
The coordinate free expression is
\[
\frac{\partial\varphi}{\partial t} = \operatorname{div}J\Delta_{g}%
\Psi(x,d\varphi).
\]

\end{proposition}

\begin{proof}
Taking $(x,v)$ for coordinates of $T^{\ast}L,$ let $y^{\alpha}$ be coordinates
in $M$, $\alpha=1, ... , 2n$. This gives a frame
\begin{equation}
e_{i}:=\frac{\partial F}{\partial x^{i}}=\frac{\partial\Psi^{\alpha}}{\partial
x^{i}}\frac{\partial}{\partial y^{\alpha}}+\frac{\partial\Psi^{\alpha}%
}{\partial v^{k}}\frac{\partial^{2}\varphi}{\partial x^{i}\partial x^{j}%
}\delta^{jk}\frac{\partial}{\partial y^{\alpha}}=\Psi_{i}+\varphi_{ij}%
\delta^{jk}\Psi_{k+n} \label{e}%
\end{equation}
where
\begin{align*}
\Psi_{i}  &  : =D_{x^{i}}\Psi=\frac{\partial\Psi^{\alpha}}{\partial x^{i}%
}\frac{\partial}{\partial y^{\alpha}}\\
\Psi_{j+n}  &  :=D_{v^{j}}\Psi=\frac{\partial\Psi^{\alpha}}{\partial v^{j}%
}\frac{\partial}{\partial y^{\alpha}}.
\end{align*}
Letting also
\[
F_{i}^{\alpha}:=\frac{\partial\Psi^{\alpha}}{\partial x^{i}}+\frac
{\partial\Psi^{\alpha}}{\partial v^{k}}\frac{\partial^{2}\varphi}{\partial
x^{i}\partial x^{j}}\delta^{jk}%
\]
we have
\[
e_{i}=F_{i}^{\alpha}\frac{\partial}{\partial y^{\alpha}}.
\]
Now suppose
\[
h=h_{\alpha\beta}dy^{a}dy^{\beta}%
\]
is the Riemannian metric on $M$. We are also assuming $\omega(V,W)=h(JV,W)$.
We compute the induced metric $g$ from the immersion:
\begin{align}
g_{ij}  &  =h(\partial_{i}F,\partial_{j}F)\label{metricg}\\
&  =\left(  \frac{\partial\Psi^{\alpha}}{\partial x^{i}}\frac{\partial
\Psi^{\beta}}{\partial x^{j}}+\sum_{k}\left(  \frac{\partial\Psi^{\alpha}%
}{\partial x^{i}}\frac{\partial\Psi^{\beta}}{\partial v^{k}}\frac{\partial
^{2}\varphi}{\partial x^{j}\partial x^{k}}+\frac{\partial\Psi^{\alpha}%
}{\partial x^{j}}\frac{\partial\Psi^{\beta}}{\partial v^{k}}\frac{\partial
^{2}\varphi}{\partial x^{i}\partial x^{k}}\right)  +\sum_{k.l}\frac
{\partial\Psi^{\alpha}}{\partial v^{k}}\partial_{ik}^{2}\varphi\frac
{\partial\Psi^{\beta}}{\partial v^{l}}\partial_{jl}^{2}\varphi\right)
h_{\alpha\beta}.\nonumber\\
&  =h(\Psi_{i}+\varphi_{ik}\Psi_{k+n},\Psi_{j}+\varphi_{jl}\Psi_{l+n}%
)\nonumber\\
&  =h(\Psi_{i},\Psi_{j})+\sum_{k}\left(  \varphi_{ik}h(\Psi_{k+n},\Psi
_{j})+\varphi_{jk}h(\Psi_{i},\Psi_{k+n})\right)  +\sum_{k,l}\varphi
_{ik}\varphi_{jl}h(\Psi_{k+n},\Psi_{l+n}).\nonumber
\end{align}
Since $\Psi:T^{\ast}L\rightarrow M$ is a symplectomorphism with $\Psi^{\ast
}\omega=dx\wedge dv$, we have
\begin{align}
\delta_{ij}  &  =dx\wedge dv(\partial/\partial x^{i},\partial/\partial
v^{j})=\Psi^{\ast}\omega\,(\partial/\partial x^{i},\partial/\partial
v^{j})\label{psi}\\
&  =h(J\Psi_{\ast}(\partial/\partial x^{i}),\Psi_{\ast}(\partial/\partial
v^{j}))=h(J\Psi_{i},\Psi_{j+n}).\nonumber
\end{align}
Similarly
\begin{equation}
h(J\Psi_{i+n},\Psi_{j+n})=\omega(\Psi_{\ast}(\partial/\partial v^{i}%
),\Psi_{\ast}(\partial/\partial v^{j}))=0. \label{psi1}%
\end{equation}
Now as $F$ describes a Lagrangian manifold, (summing repeated indices below)
\begin{align}
0  &  =\omega(e_{i},e_{j})=h(J\Psi_{i}+\partial_{k}\varphi_{i}J\Psi_{k+n}%
,\Psi_{j}+\partial_{l}\varphi_{j}\Psi_{l+n})\label{psi2}\\
&  =h(J\Psi_{i},\Psi_{j})+\partial_{l}\varphi_{j}h(J\Psi_{i},\Psi
_{l+n})+\partial_{k}\varphi_{i}h(J\Psi_{k+n},\Psi_{j})+\partial_{k}\varphi
_{i}\partial_{l}\varphi_{j}h(J\Psi_{k+n},\Psi_{l+n})\nonumber\\
&  =h(J\Psi_{i},\Psi_{j})-\varphi_{jk}h(\Psi_{i},J\Psi_{k+n})+\varphi
_{ik}h(J\Psi_{k+n},\Psi_{j})+\varphi_{ik}\varphi_{jl}h(J\Psi_{k+n},\Psi
_{l+n})\nonumber\\
&  =h(J\Psi_{i},\Psi_{j})+\varphi_{ik}\varphi_{jl}h(J\Psi_{k+n},\Psi
_{l+n})\hspace{5cm}\mbox{by \eqref{psi}}\nonumber\\
&  =h(J\Psi_{i},\Psi_{j}).\nonumber
\end{align}
Now, $\{\Psi_{i},J\Psi_{j}:1\leq i,j\leq n\}$ is a basis for the ambient
tangent space at a point in the image of $F$. So is $\{\Psi_{i},\Psi
_{j+n}:1\leq i,j\leq n\}$ (as $\Psi$ is a local diffeomorphism). We represent
the latter vectors by
\begin{equation}
\Psi_{i+n}=a^{ij}\Psi_{j}+b^{ij}J\Psi_{j}. \label{psi3}%
\end{equation}
\ifthenelse{\boolean{bluecomments}}{\textcolor{blue}{
Note that by (\ref{psi})
\begin{align*}
\delta_{ij}  &  =h(J\Psi_{j},\Psi_{i+n}) =h\left(  J\Psi_{j},a^{ik}\Psi
_{k}+b^{ik}J\Psi_{k}\right) \\
&  =a^{ik}h\left(  J\Psi_{j},\Psi_{k}\right)  +b^{ik}h\left(  J\Psi_{j},J\Psi_{k}\right) \\
&  =a^{ik}\omega\left(  \Psi_{j},\Psi_{k}\right)  +b^{ik}h_{jk}=b^{ik}h_{jk}.
\end{align*}}}{} Computing the pairing $h\left(  J\Psi_{j},\Psi_{i+n}\right)
$ using \eqref{psi} on the left and (\ref{psi3}) on the right yields
$b^{ij}=h^{ij}$ as the inverse of the positive definite matrix $h_{ij}%
:=h(\Psi_{i},\Psi_{j})$. Now recalling (\ref{F})
\[
\frac{\partial F}{\partial t}=\left(  \frac{\partial}{\partial t}%
\frac{\partial\varphi}{\partial x^{k}}\right)  \frac{\partial\Psi}{\partial
v^{k}},
\]
project onto the normal space:%
\begin{align*}
\left(  \frac{\partial F}{\partial t}\right)  ^{\perp}  &  =\frac
{\partial\varphi_{t}}{\partial x^{k}}h(\Psi_{k+n},Je_{p})Je_{q}g^{pq}\\
&  =\frac{\partial\varphi_{t}}{\partial x^{k}}h(\Psi_{k+n},J\Psi_{p}%
+\varphi_{pj}J\Psi_{j+n})Je_{q}g^{pq}\\
&  =\frac{\partial\varphi_{t}}{\partial x^{k}}h(\Psi_{k+n},J\Psi_{p}%
)Je_{q}g^{pq}\\
&  =\frac{\partial\varphi_{t}}{\partial x^{k}}\delta_{kp}Je_{q}g^{pq}%
\,\,\,\,\,\,{\mbox{ by \eqref{psi}}}\\
&  =\frac{\partial\varphi_{t}}{\partial x^{k}}\left(  Je_{q}g^{kq}\right) \\
&  =J\nabla\varphi_{t}.
\end{align*}
Let $H=H^{m}Je_{m}$ for the Lagrangian $L$. As $JH$ is tangential its
divergence on $L$ is
\begin{align}
\operatorname{div}(JH)  &  =-\operatorname{div}\left(  H^{m}e_{m}\right)
\label{div part 1}\\
&  =-g^{ab}h(\nabla_{e_{a}}\left(  H^{m}e_{m}\right)  ,e_{b})\nonumber\\
&  =-g^{ab}h\left(  \frac{\partial}{\partial x^{a}}H^{m}e_{m}+H^{m}\Gamma
_{am}^{p}e_{p},e_{b}\right) \nonumber\\
&  =-g^{ab}\frac{\partial}{\partial x^{a}}H^{m}g_{mb}-g^{ab}H^{m}\Gamma
_{am}^{p}g_{pb}\nonumber\\
&  =-\frac{\partial}{\partial x^{a}}H^{a}-H^{m}\Gamma_{am}^{a}\nonumber
\end{align}
where the Christoffel symbols are for the induced metric $g$. The components
of $H$ are given by
\begin{equation}
H^{a}=h\left(  H,Je_{p}\right)  g^{ap}=h\left(  g^{ij}\left(  \frac
{\partial^{2}F^{\beta}}{\partial x^{i}\partial x^{j}}+F_{i}^{\alpha}%
F_{j}^{\gamma}\tilde{\Gamma}_{\alpha\gamma}^{\beta}\right)  \frac{\partial
}{\partial y^{\beta}},Je_{p}\right)  g^{ap}. \label{ha_1}%
\end{equation}
Now differentiate components of (\ref{e})
\[
\frac{\partial^{2}F^{\beta}}{\partial x^{i}\partial x^{j}}=\frac{\partial
\Psi^{\beta}}{\partial x^{j}\partial x^{i}}+\frac{\partial^{3}\varphi
}{\partial x^{j}\partial x^{i}\partial x^{k}}\delta^{kl}\frac{\partial
\Psi^{\beta}}{\partial v^{l}}+\frac{\partial^{2}\varphi}{\partial
x^{i}\partial x^{k}}\delta^{kl}\frac{\partial^{2}\Psi^{\beta}}{\partial
x^{j}\partial v^{l}}.
\]
Plug in (\ref{e}) to get
\begin{align*}
h  &  \left(  \frac{\partial^{2}F^{\beta}}{\partial x^{i}\partial x^{j}}%
\frac{\partial}{\partial y^{\beta}},Je_{p}\right)  =\omega\left(  e_{p}%
,\frac{\partial^{2}F^{\beta}}{\partial x^{i}\partial x^{j}}\frac{\partial
}{\partial y^{\beta}}\right) \\
=  &  \omega\left(  \frac{\partial\Psi^{\delta}}{\partial x^{p}}\frac
{\partial}{\partial y^{\delta}}+\frac{\partial^{2}\varphi}{\partial
x^{p}\partial x^{q}}\delta^{qm}\frac{\partial\Psi^{\delta}}{\partial v^{m}%
}\frac{\partial}{\partial y^{\delta}}, \frac{\partial\Psi^{\beta}}{\partial
x^{j}\partial x^{i}}\frac{\partial}{\partial y^{\beta}}+\frac{\partial
^{3}\varphi}{\partial x^{j}\partial x^{i}\partial x^{k}}\delta^{kl}%
\frac{\partial\Psi^{\beta}}{\partial v^{l}}\frac{\partial}{\partial y^{\beta}%
}+\frac{\partial^{2}\varphi}{\partial x^{i}\partial x^{k}}\delta^{kl}%
\frac{\partial^{2}\Psi^{\beta}}{\partial x^{j}\partial v^{l}}\frac{\partial
}{\partial y^{\beta}}\right) \\
=  &  \frac{\partial^{3}\varphi}{\partial x^{j}\partial x^{i}\partial x^{k}%
}\delta^{kl}\omega\left(  \frac{\partial\Psi^{\delta}}{\partial x^{p}}%
\frac{\partial}{\partial y^{\delta}},\frac{\partial\Psi^{\beta}}{\partial
v^{l}}\frac{\partial}{\partial y^{\beta}}\right)  +\frac{\partial^{2}\varphi
}{\partial x^{p}\partial x^{q}}\delta^{qm}\frac{\partial^{3}\varphi}{\partial
x^{j}\partial x^{i}\partial x^{k}}\delta^{kl}\omega\left(  \frac{\partial
\Psi^{\delta}}{\partial v^{m}}\frac{\partial}{\partial y^{\delta}}%
,\frac{\partial\Psi^{\beta}}{\partial v^{l}}\frac{\partial}{\partial y^{\beta
}}\right) \\
&  +F_{p}^{\delta}\left(  \frac{\partial\Psi^{\beta}}{\partial x^{j}\partial
x^{i}}+\frac{\partial^{2}\varphi}{\partial x^{i}\partial x^{k}}\delta
^{kl}\frac{\partial^{2}\Psi^{\beta}}{\partial x^{j}\partial v^{l}}\right)
\omega_{\delta\beta}\circ F\\
=  &  \frac{\partial^{3}\varphi}{\partial x^{j}\partial x^{i}\partial x^{k}%
}\delta^{kl}\Psi^{\ast}\omega\left(  \frac{\partial}{\partial x^{p}}%
,\frac{\partial}{\partial v^{l}}\right)  +\frac{\partial^{2}\varphi}{\partial
x^{p}\partial x^{q}}\delta^{qm}\frac{\partial^{3}\varphi}{\partial
x^{j}\partial x^{i}\partial x^{k}}\delta^{kl}\Psi^{\ast}\omega\left(
\frac{\partial}{\partial v^{m}},\frac{\partial}{\partial v^{l}}\right) \\
&  +F_{p}^{\delta}\left(  \frac{\partial\Psi^{\beta}}{\partial x^{j}\partial
x^{i}}+\frac{\partial^{2}\varphi}{\partial x^{i}\partial x^{k}}\delta
^{kl}\frac{\partial^{2}\Psi^{\beta}}{\partial x^{j}\partial v^{l}}\right)
\omega_{\delta\beta}\circ F\\
=  &  \frac{\partial^{3}\varphi}{\partial x^{j}\partial x^{i}\partial x^{k}%
}\delta^{kl}\delta_{pl}+F_{p}^{\delta}\left(  \frac{\partial\Psi^{\beta}%
}{\partial x^{j}\partial x^{i}}+\frac{\partial^{2}\varphi}{\partial
x^{i}\partial x^{k}}\delta^{kl}\frac{\partial^{2}\Psi^{\beta}}{\partial
x^{j}\partial v^{l}}\right)  \omega_{\delta\beta}\circ F.
\end{align*}
Now also
\begin{equation}
h\left(  \frac{\partial}{\partial y^{\beta}},Je_{p}\right)  =\omega\left(
F_{p}^{\delta}\frac{\partial}{\partial y^{\delta}},\frac{\partial}{\partial
y^{\beta}}\right)  =F_{p}^{\delta}\omega_{\delta\beta}\circ F.\nonumber
\end{equation}
Combining (\ref{ha_1}) and the above
\begin{align*}
H^{a}  &  =g^{ij}g^{ap}\left(  \frac{\partial^{3}\varphi}{\partial
x^{j}\partial x^{i}\partial x^{k}}\delta^{kl}\delta_{pl}+F_{p}^{\delta}\left(
\frac{\partial\Psi^{\beta}}{\partial x^{j}\partial x^{i}}+\frac{\partial
^{2}\varphi}{\partial x^{i}\partial x^{k}}\delta^{kl}\frac{\partial^{2}%
\Psi^{\beta}}{\partial x^{j}\partial v^{l}}\right)  \omega_{\delta\beta}\circ
F\right) \\
&  +g^{ij}g^{ap}F_{i}^{\alpha}F_{j}^{\gamma}\tilde{\Gamma}_{\alpha\gamma
}^{\beta}F_{p}^{\delta}\omega_{\delta\beta}\circ F.
\end{align*}
Thus using the expression we derived in (\ref{div part 1})
\begin{align}
\operatorname*{div}  &  (JH) =-g^{ap}g^{ij}\frac{\partial^{4}\varphi}{\partial
x^{a}\partial x^{j}\partial x^{i}\partial x^{p}}-\left(  \frac{\partial
}{\partial x^{a}}\left(  g^{ij}g^{ap}\right)  \right)  \frac{\partial
^{3}\varphi}{\partial x^{j}\partial x^{i}\partial x^{p}}\label{divJH}\\
&  -\frac{\partial}{\partial x^{a}}\left(  g^{ij}g^{ap}\right)  F_{p}^{\delta
}\left[  \left(  \frac{\partial\Psi^{\beta}}{\partial x^{j}\partial x^{i}%
}+\frac{\partial^{2}\varphi}{\partial x^{i}\partial x^{k}}\delta^{kl}%
\frac{\partial^{2}\Psi^{\beta}}{\partial x^{j}\partial v^{l}}\right)
+F_{i}^{\alpha}F_{j}^{\gamma}\tilde{\Gamma}_{\alpha\gamma}^{\beta}\right]
\omega_{\delta\beta}\circ F-H^{m}\Gamma_{am}^{a}.\nonumber
\end{align}
Now recalling (\ref{metricg}), we see the metric components $g_{ab}$ involve
second order derivatives in terms of $\varphi,$ thus $\Gamma_{ij}^{k}$ are
third order. \ So each term above after the first term is at most third order.
\end{proof}

\subsection{Short time existence}

\bigskip

\begin{proposition}
\label{stexist}Given an initial smooth immersion of a compact $L\rightarrow
M,$ there exists a solution to (\ref{flow0},\ref{flow0b}) for some short time.
\end{proposition}

\begin{proof}
Choose a Weinstein neighborhood containing $L.$ \ Now suppose we have
$\varphi$ which satisfies the fourth order equation
\begin{align}
&  \varphi_{t} =-g^{ap}g^{ij}\frac{\partial^{4}\varphi}{\partial x^{a}\partial
x^{j}\partial x^{i}\partial x^{p}}+G(x,D\varphi,D^{2}\varphi,D^{3}%
\varphi)=\operatorname{div}(JH)\label{MM}\\
&  \varphi(\cdot,0) =0.\nonumber
\end{align}
Then the immersions $F$ generated from $\varphi(x,t)$ satisfy
\begin{align*}
\left(  \frac{\partial F}{\partial t}\right)  ^{\perp}  &  =J\nabla\varphi_{t}
=J\nabla\operatorname{div}(JH).
\end{align*}
As the normal component satisfies the appropriate equation, we may compose
with diffeomorphisms to get a flow (see Claim \ref{moddif} below) such that%
\begin{equation}
\frac{\partial F}{\partial t}=J\nabla\operatorname{div}(JH). \label{hflow}%
\end{equation}
Now the equation (\ref{MM}) is precisely of the form of $2p$ order quasilinear
parabolic equation studied in \cite{MM}. By \cite[Theorem 1.1]{MM} we have
short time existence for the solution to (\ref{MM}), thus we have short-time
existence for the flow (\ref{hflow}).
\end{proof}

\subsection{Uniqueness}


We start with a standard observation.

\begin{claim}
\label{moddif}Suppose that $F:L\times\lbrack0,T)\rightarrow M$ is a family of
immersions satisfying
\[
\left(  \frac{\partial F}{\partial t}\right)  ^{\perp}=N(x,t)
\]
for some vector field $N(x,t)$ which is normal to the immersed submanifold
$F(\cdot,t)(L)$. There exists a unique family of diffeomorphisms $\chi
_{t}:L\rightarrow L$ such that
\[
\frac{\partial}{\partial t}F(\chi_{t}(x),t) =N(\chi_{t}(x),t) \,\,\,\,{and}%
\,\,\,\, \chi_{0} =Id_{|L}.
\]

\end{claim}

\begin{proof}
Given the flow exists, the given velocity field will decompose orthogonally
into normal and tangential components:
\[
\frac{\partial F}{\partial t}=N(x,t)+T\left(  x,t\right)  .
\]
Consider the time-dependent vector field on $L$%
\[
V(x,t)=-D_{L}F(x,t)^{-1}T\left(  x,t\right)
\]
By the Fundamental Theorem on Flows, (cf. \cite[Theorem 9.48]{Jack}) there is
a unique flow on $L$ starting at the identity and satisfying%
\[
\frac{\partial}{\partial t}\chi_{t}(x)=V(\chi_{t}(x),t).
\]
Composing this flow with the original flow $F$ yields the result.
\end{proof}

\begin{theorem}
\label{thm31}The solution to the initial value problem \eqref{flow0} -
\eqref{flow0b} is unique. More precisely, if $F_{1}$ and $F_{2}$ are two
solutions of \eqref{flow0} such that $F_{1}(x,t_{0})=F_{2}(x,t_{0}^{\prime})$
for some $t_{0} ,t_{0}^{\prime}$ and all $x\in L$, then
\[
F_{1}(x,t_{0}+\tau)=F_{2}(x,t_{0}^{\prime}+\tau)
\]
for all $\tau$ in an open neighborhood of $0$ where both sides above are defined.
\end{theorem}

\begin{proof}
Without loss of generality, we take $t_{0}=t_{0}^{\prime}=0$. Let
$L=F_{1}(\cdot,0)=F_{2}(\cdot,0)$ and $\Psi:U\subset T^{\ast}L\to V\subset M$
be a Lagrangian neighborhood mapping.

First, we show that the normal flow of $F_{i}(\cdot,t)$ is given in the
neighborhood $\Psi$ by the graph of an exact section $d\varphi_{i}(\cdot,t)$
where $\varphi_{i}$ solves a problem of the form (\ref{MM}). To this end, note
that for $\tau$ in the domain, the path $\left\{  F_{i}(\cdot,t),t\in
\lbrack0,\tau]\text{ }\right\}  $ is a Hamiltonian isotopy between
$F_{i}(\cdot,0)$ and $F_{i}(\cdot,\tau)$. Being a Hamiltonian isotopy is
invariant under the symplectomorphism $\Psi$, so the sections $\Psi
^{-1}\left(  F_{i}(\cdot,0)\right)  $ and $\Psi^{-1}\left(  F_{i}(\cdot
,\tau)\right)  $ are Hamiltonian isotopic. By \cite[Corollary 6.2]{Weinstein},
Lagrangian submanifolds that are near to the $0$-section are given as graphs
of closed sections of the cotangent bundle. As the flow is smooth, for small
times the Lagrangian submanifolds are near enough to be described by closed
sections. According to \cite[Proposition 9.4.2]{McDuffSalamon}, these sections
are exact, that is
\[
\Psi^{-1}\circ F_{i}(\cdot,\tau)\left(  L\right)  =\left\{  d\varphi
_{i}\left(  x,\tau\right)  :x\in L\right\}  .
\]
In other words,
\[
\Psi(\left\{  d\varphi_{i}\left(  x,\tau\right)  :x\in L_{i}\right\}
)=F_{i}(\cdot,\tau)\left(  L\right)
\]
meaning that for each $\tau$
\[
\Psi\circ d\varphi_{i}:L\rightarrow T^{\ast}L\rightarrow M
\]
is a Lagrangian immersion, which may have reparameterized the base. In
particular, the flow $F_{i}$ determines a flow of scalar functions
$\varphi_{i},$ which recovers the same family of submanifolds at $F_{i}$ (up
to reparameterization) as do $\Psi\circ d\varphi_{i}\left(  \cdot,t\right)  $.
By Proposition \ref{flow equation}, the scalar equation (\ref{MM}) holds on
$\varphi_{i}$ as does the initial condition $d\varphi_{i}\equiv0$. It follows
that $\varphi_{1}$ and $\varphi_{2}$ both satisfy the same equation (\ref{MM})
and have the same initial condition, so $\varphi_{1}=\varphi_{2}+C$ for some
constant $C$. Thus the flows $F_{1}$ and $F_{2}$ are the same.
\end{proof}

Theorem \ref{thm31} allows for seamless extension of the flow: While the
Weinstein's Lagrangian neighborhood may only exist around $L_{0}$, if another
Lagrangian neighborhood of $L_{0}$ extends the flow, the two flows patch
together smoothly.

\section{Higher order estimates based on curvature bounds}

The goal of this section is to show that a solution with uniformly bounded
second fundamental form over $[0,T)$ enjoy estimates of all orders and can be extended.

\begin{theorem}
\label{c1}Suppose that the flow (\ref{flow0}) exists on $[0,T)$ and the second
fundamental form has a uniform bound on $[0,T)$. Then the flow converges
smoothly as $t\rightarrow T$ so can be extended to $[0,T+\varepsilon)$ for
some $\varepsilon>0$.
\end{theorem}

To prove this theorem, it is essential in our approach to establish a-priori
estimates from the integral estimates derived from the following differential inequality:

\begin{proposition}
\label{cp11}Suppose that $F$ is a solution to (\ref{flow0}) on $[0,T)$ for a
compact Lagrangian submanifold $L$ inside a compact $M$. Suppose the second
fundamental form has a uniform bound $K$. There exists $C$ depending on $K$,
the ambient geometry of $M$ and $\operatorname{Vol}(L_{0})$ such that for all
$k\geq2$%
\begin{equation}
\frac{d}{dt}\int_{L}\left\vert \nabla^{k-1}A\right\vert ^{2}dV_{g}(t)\leq
C\int_{L}\left\vert \nabla^{k-1}A\right\vert ^{2}dV_{g}(t)+C\sum_{l=0}%
^{k-2}\int_{L}\left\vert \nabla^{l}A\right\vert ^{2}dV_{g}(t).
\label{cp11result}%
\end{equation}

\end{proposition}

A Weinstein neighborhood map determines the equation the flow must satisfy,
and we could derive estimates of all orders based on this particular
equation.
However, the flow is expected to leave a given neighborhood after some time,
and we will need to take a new neighborhood. We would need to know the speed
of the flow to patch estimates from one neighborhood to another, but this
requires knowing the size of the Weinstein neighborhoods around the Lagrangian
submanifolds at different times.

We require charts with uniform geometric estimates. To obtain these we appeal
to uniform local Darboux coordinates given in \cite{JLS}. These charts are
local but are given with uniform geometric bounds. The short-time existence of
the flow is already determined by the global Weinstein neighborhoods; we write
the flow in these Darboux charts as a scalar equation from which we derive
integral estimates for derivatives of any order.

\subsection{Uniform Darboux charts.}

We record \cite[Prop.3.2 and Prop.3.4]{JLS} on existence of Darboux
coordinates with estimates on a symplectic manifold. Let $\pi:\mathcal{U}%
\rightarrow M$ be the $U(n)$ frame bundle of $M$. A point in $\mathcal{U}$ is
a pair $(p,v)$ with $\pi(p,v)=p\in M$ and $v:\mathbb{R}^{2n}\rightarrow
T_{p}M$ an isomorphism satisfying $v^{\ast}(\omega_{p})=\omega_{0}$ and
$v^{\ast}(h|_{p})=h_{0}$ (the standard metric on $\mathbb{C}^{n}$). The right
action of $U(n)$ on $\mathcal{U}$ is free: $\gamma(p,v)=(p,v\circ\gamma)$ for
any $\gamma\in U(n)$.

\begin{proposition}
[Joyce-Lee-Schoen]\label{Darboux} Let $(M,\omega)$ be a real $2n$-dimensional
symplectic manifold without boundary, and a Riemannian metric $h$ compatible
with $\omega$ and an almost complex structure $J$. Let $\mathcal{U}$ be the
$U(n)$ frame bundle of $M$. Then for small $\varepsilon>0$ we can choose a
family of embeddings $\Upsilon_{p,v}:B^{2n}_{\varepsilon}\rightarrow M$
depending smoothly on $\left(  p,v\right)  \in U$, where $B^{2n}_{\varepsilon
}$ is the ball of radius $\varepsilon$ about $0$ in $\mathbb{C}^{n},$ such
that for all $\left(  p,v\right)  \in U$ we have

\begin{enumerate}
\item $\Upsilon_{p,v}(0)=p$ and $d\Upsilon_{p,v}|_{0}=v:\mathbb{C}%
^{n}\rightarrow T_{p}M;$

\item $\Upsilon_{p,v\circ\gamma}(0)\equiv\Upsilon_{p,v}\circ\gamma$ for all
$\gamma\in U(n);$

\item $\Upsilon_{p,v}^{\ast}(\omega)\equiv\omega_{0}=\frac{\sqrt{-1}}{2}
\sum_{j=1}^{n}dz_{j}\wedge d\bar{z}_{j};$ and

\item $\Upsilon_{p,v}^{\ast}(h)=h_{0}+O(|z|)$.
\end{enumerate}

Moreover, for a dilation map $\mathbf{t}:B^{2n}_{R}\to B^{2n}_{\varepsilon}$
given by $\mathbf{t}(z)=tz$ where $t\leq\varepsilon/R$, set $h^{t}_{p,v} =
t^{-2}(\Upsilon_{p,v}\circ\mathbf{t})^{*}h$. Then it holds

\begin{enumerate}
\item[(5)] $\|h^{t}_{p,v}-h_{0}\|_{C^{0}} \leq C_{0} t \ \ \ \ \mbox{and}
\ \ \ \|\partial h^{t}_{p,v}\|_{C^{0}}\leq C_{1} t$,
\end{enumerate}

where norms are taken w.r.t. $h_{0}$ and $\partial$ is the Levi-Civita
connection of $h_{0}$.
\end{proposition}

\begin{proposition}
\label{JLScharts}\label{charts}Suppose that $M$ is a compact symplectic
manifold with a compatible Riemannian metric $h$. Suppose that $L$ is a
compact Lagrangian submanifold of $M$ with second fundamental form bounded
above by $K$ and volume bounded above by $V_{0}$. Given $c_{n}>0,$ there
exists an $r_{0}=r_{0}(K,c_{n})>0$ and a finite cover of $L$ by Darboux charts
$\Upsilon_{p_{i},v_{i}}$:$B_{r_{0}}^{2n}\rightarrow M$ centered at points
$p_{i}$ on $L$ such that

\begin{enumerate}
\item The connected component of $L\cap B_{r_{0}}^{2n}$ containing $p_{j}$ is
represented by a graph $\left(  x,d\varphi^{(j)}\right)  $ over $B_{r_{0}%
}^{2n}\cap\mathbb{R} ^{n}\times\left\{  0\right\}  $ for some potential
$\varphi^{(j)}$.

\item The tangent plane at each point of this connected component satisfies a
closeness condition with respect to the planes $v_{i}(\mathbb{R}^{n}%
\times\{0\}):$
\begin{equation}
\max_{\substack{\left\vert e\right\vert _{g}=1,\text{ }e\in T_{p}L\\\left\vert
\nu\right\vert _{\delta_{0}}=1,\nu\in\left\{  0\right\}  \times\mathbb{R}^{n}%
}}e\cdot\nu<c_{n} \label{tangentspaceclose}%
\end{equation}
where the dot product is in the euclidean metric $\delta_{0}$, and $c_{n}$ is
a small universal constant (say $c_{n}=\frac{1}{10\sqrt{n}}$) chosen so that
quantities such as the volume element and coordinate expression for $h$ are
bounded by universal constants.

\item The ambient metric $h$ is very close to the euclidean metric, that is
$\|h-\delta_{0}\| < c_{n}$ for some $c_{n}$ (can be the same $c_{n}$ as in (2) above).

\item The submanifold $L$ is covered by the charts obtained by restricting
these charts to $B_{r_{0}}^{2n}( p_{j}) $

\item The number of such points $\left\{  p_{j}\right\}  $ satisfies
\begin{equation}
{N}(K,V_{0})\leq\frac{C(n)V_{0}}{r_{0}^{n}(K, c_{n})}. \label{numberbounded}%
\end{equation}

\end{enumerate}
\end{proposition}

\begin{proof}
At each point $p\in L$ we take a Darboux chart $\Upsilon_{p,v}$ as described
above with that $T_{p}L=\mathbb{R}^{n}\times\left\{  0\right\}  $ in the given
chart. Note that after some fixed re-scalings, we can assert via Proposition
\ref{Darboux} that $\Upsilon_{p,v}$ exists on $B_{\varepsilon_{0}}^{2n}$ and
satisfies any near euclidean metric conditions we choose to prescribe,
including the closeness condition: $\left\vert h-\delta_{0}\right\vert < c$.
Now we may apply Proposition \ref{A-bound} which asserts existence of a ball
$B_{r_{0}} ^{n}(p)\subset\mathbb{R}^{n}\times\{0\}$ over which $L$ is
representable as a graph, with (\ref{tangentspaceclose}) holding. The quantity
$r_{0}$ will depend on $K$.

Consider the compact immersed submanifold $L$ as a metric space $(L,d)$.
\ifthenelse{\boolean{bluecomments}}{\textcolor{blue}{We don't want to use ambient metric as nearby pieces could get caught in each other's metric balls. }}{}
Taking a finite cover of metric balls $B_{r_{0}/4}(p)$ for $p\in L$ and
applying Vitalli's covering Lemma, we conclude that there is a subset of these
points $\{p_{j}\} $ so that $L=\cup_{i}B_{3_{0}/4}(p_{i})$ and $B_{r_{0}%
/4}(p_{i})$ are mutually disjoint. By (\ref{tangentspaceclose}), $L\cap
B_{3r_{0}/4}(p_{i})$ is in the image of a graph given by Proposition
\ref{A-bound}. In particular, the disjoint $B_{r_{0}/4}(p_{i})$'s have a
minimum total volume $\omega_{n}c^{n} r_{0}^{n}$. The bound
(\ref{numberbounded}) on the number of balls follows. As $\{B_{3 r_{0}%
/4}(p_{i})\}$ covers $L$ and each of these balls is contained in a graph over
$B_{ r_{0}}^{n},$ we take the set of the graphs as the cover.
\end{proof}

The scalar functions from the the exact sections of $T^{*}L$ are globally
defined on $L$ via the abstract Weinstein map $\Psi$. We have utilized them to
establish short-time existence and uniqueness for our geometric flow of $F$.
However, for higher order a-priori estimates, we need to set up the flow
equation in a Darboux chart with estimates on the metric as described above.
Fortunately, each $\Upsilon_{p,v}$ is a symplectomorphism, which takes
gradient graphs $(x,d\varphi)$ to Lagrangian submanifolds, so the computations
in section \ref{existence} can be repeated verbatim, with $\Upsilon_{p,v}$ in
place of $\Psi$. In particular, in each chart, the flow is determined by an
equation
\begin{equation}
\varphi_{t}=-g^{ap}g^{ij}\frac{\partial^{4}\varphi}{\partial x^{a}\partial
x^{j}\partial x^{i}\partial x^{p}}+G(x,D\varphi,D^{2}\varphi,D^{3}\varphi).
\label{MM1}%
\end{equation}

\begin{remark}
\label{remark1}A precise computation in Darboux coordinates of the expression
(\ref{ha_1}) gives
\begin{align*}
h\left(  H,Je_{p}\right)   &  =g^{ij}\varphi_{pij}+g^{ij}\tilde{\Gamma}%
_{ij}^{p+n} +g^{ij}\varphi_{kj}\delta^{km}\tilde{\Gamma}_{i,m+n}^{p+n}%
+g^{ij}\varphi_{ki}\delta^{km}\tilde{\Gamma}_{m+n,j}^{p+n}\\
&  +g^{ij}\varphi_{ki}\varphi_{lj}\delta^{km}\delta^{lr}\tilde{\Gamma
}_{m+n,r+n}^{p+n}-g^{ij}\tilde{\Gamma}_{ij}^{q}\varphi_{pq}-g^{ij}\varphi
_{kj}\delta^{km}\tilde{\Gamma}_{i,m+n}^{q}\varphi_{pq}\\
&  -g^{ij}\varphi_{ki}\delta^{km}\tilde{\Gamma}_{m+n,j}^{q}\varphi_{pq}
-g^{ij}\varphi_{ki}\varphi_{lj}\delta^{km}\delta^{lr}\tilde{\Gamma}
_{m+n,r+n}^{q}\varphi_{pq}%
\end{align*}
where $\tilde{\Gamma}_{ij}^{q}$ are Christoffel symbols in the ambient metric
$\left(  M,h\right)  $. Considering that each expression of the form $g^{ab}$
is a smooth function in terms of $D^{2}\varphi$ with dependence on zero order
of $h$ and each $\tilde{\Gamma}_{ij}^{\beta}$ expression depends on $Dh$ and
$h$, one may conclude (after computing $\operatorname{div}(JH)$ as in
(\ref{divJH})) that $G$ can be written as a sum of expressions that are

\begin{enumerate}
\item quadratic in $D^{3}\varphi$ and smooth in $D^{2}\varphi,h$ in a
predetermined way

\item linear in $D^{3}\varphi$, smooth in $D^{2}\varphi,h,Dh$ in a
predetermined way

\item smooth in $D^{2}\varphi,h$ and linear in $D^{2}h$ in a predetermined way

\item smooth in $D^{2}\varphi,h,Dh$ in a predetermined way.
\end{enumerate}
\end{remark}

This allows us to make a claim that there is uniform control on the important
quantities involved in the equation we are solving.


\begin{proposition}
\label{JLScharts2} Suppose that $L$ is a compact Lagrangian manifold with
volume $V_{0}$ and evolves by \eqref{flow0} on $[0,T)$. If the norm of second
fundamental $A$ of $L_{t}$ satisfies $|A|_{g(t)}\leq K$ for $t\in[0,T)$, then
after a fixed rescaling on $M$ there is a finite set of Darboux charts such that

\begin{enumerate}
\item The submanifold is covered by graphs over $B_{1}^{2n}\cap\mathbb{R}
^{n}\times\left\{  0\right\}  $.

\item The submanifold is graphical over $B_{5}^{2n}\cap\mathbb{R}^{n}%
\times\left\{  0\right\}  $ in each chart.

\item The slope bound (\ref{tangentspaceclose}) holds over $B_{5}^{n}(0).$

\item The flow (\ref{flow0}) is governed by (\ref{MM1}) locally in these charts.

\item For each chart, the $G$ from (\ref{MM1}) satisfies a uniform bound on
any fixed order derivatives of $G$ (in terms of all four arguments, not with
respect to $x$ coordinate before embedding.)

\item The number of charts is controlled
\begin{equation}
{N}(K,V_{0})\leq C(K)\frac{V_{0}}{r_{0}^{n}(K)}. \label{howmanycharts}%
\end{equation}

\end{enumerate}
\end{proposition}

\begin{proof}
Rescale $M$ so that $r_{0} = 5.$ Then the expression $G$ becomes predictably
controlled by Remark \ref{remark1}. Choosing a cover with interior balls, as
in the proof of Proposition \ref{JLScharts}, determines the necessary number
of balls.
\end{proof}


\subsection{Localization}

Let $L_{t}$ evolve by (\ref{flow0}) \ with time $t\in\lbrack0,T),$ and assume
$\left\vert A\right\vert _{g(t)}\leq K$ for all $L_{t}.$ Our goal is to
establish integral bounds for $\left\vert \nabla^{l}A\right\vert _{g(t)}^{2},$
which only depend on $k,K,M$ and the initial volume $V_{0}$ of $L_{0}.$ To
derive the differential inequality (\ref{cp11result}) at any time $t_{0}$, we
use Proposition \ref{JLScharts2} and express geometric quantities
$g,A,\nabla^{l}A$, etc., in the (no more than $N)$ Darboux charts in terms of
$\varphi(x,t)$ for $x\in B_{5}^{n}.$ \ By compactness of $L$ and smoothness of
the flow, the flow will continue to be described by graphs of $d\varphi(x,t)$
in this open union of $N$ charts for $t\in\lbrack t_{0},t_{1})$ for some
$t_{1}>t_{0}.$ \ 


To be precise, each of the Darboux charts in Lemma \ref{JLScharts2} has a
product structure; we may assume that each chart contains coordinates
$B_{4}^{n}(0)\times B_{2}^{n}(0)$ so that $L$ is graphical over $B_{5}^{n}(0)$
and further that the collection of $B_{1}^{n}(0)\times B_{1}^{n}(0)$ covers
$L$. Now we may fix once and for all a function $\eta$ which is equal to $1$
on $B_{1}^{n}(0)\times B_{1}^{n}(0)$ and vanishes within $B_{2}^{n}(0)\times
B_{2}^{n}(0)$. For a given chart $\Upsilon^{\alpha}$ (here $\alpha\in\left\{
1,..N\right\}  $ indexes our choice of charts) we call the function
$\eta_{\alpha}$. This function will have uniformly bounded dependence on the
variables $x$ and $y$ in the chart.

Now once these $\eta_{\alpha}$ are chosen, we may then define a partition of
unity for the union of charts which form a tubular neighborhood of $L$, which
will restrict to a partition of unity for small variations of $L$:
\begin{equation}
\rho_{\alpha}^{2}:=\frac{\eta_{\alpha}^{2}}{\sum\eta_{\alpha}^{2}}.
\label{bump}%
\end{equation}
By compactness of the unit frame bundle and the smoothness of the family of
charts defined in\ Proposition \ref{Darboux}, the transition functions between
charts will have bounded derivatives to any order. Thus, in a fixed chart,
where a piece of $L$ is represented as $\left\{  \left(  x,d\varphi(x)\right)
:x\in B_{1}(0)\right\}  $, the dependence of $\rho_{\alpha}^{2}$ will be
uniformly controlled in terms of these variables, so there is a uniform
pointwise bound
\begin{equation}
\left\vert D_{x}^{2}\rho_{\alpha}^{2}\right\vert \leq C(D^{3}\varphi
,D^{2}\varphi, D\varphi, x) \label{rhobound}%
\end{equation}
were this dependence is at most linear on $D^{3}\varphi$. We will be using the
$x$ coordinates as charts for $L.$

Note also that, if we have a uniform bound on $\frac{d}{dt}D\varphi$ and
$\frac{d}{dt}D^{2}\varphi$ we can conclude a positive lower bound on
$t_{1}-t_{0};$ the flow will be described by graphs of $\varphi(x,t)$ in these
$N$ charts, and the condition (\ref{tangentspaceclose}) will be satisfied for
a slightly larger $c_{n}^{\prime}$ (say $c_{n}^{\prime}=\frac{1}{5\sqrt{n}}$
instead of $c_{n}=\frac{1}{10\sqrt{n}})$.

\subsubsection{Expression for metric and second fundamental form}

In the Darboux charts for $M$, the manifold $L$ is expressed graphically over
the $x$ coordinate via
\[
x\mapsto F(x)=\left(  x,d\varphi(x)\right)  .
\]
Thus we have a tangential frame:
\begin{equation}
e_{i}=\partial_{x^{i}}F=E_{i}+\varphi_{ik}\delta^{km}E_{m+n} \label{bdd}%
\end{equation}
with
\[
g_{ij}=h_{ij}+\varphi_{ik}\delta^{km}\varphi_{jl}\delta^{lr}h_{\left(
m+n\right)  \left(  l+n\right)  }+\varphi_{jl}\delta^{lr}h_{\left(  i\right)
\left(  l+n\right)  }+\varphi_{ik}\delta^{km}h_{\left(  m+n\right)  \left(
j\right)  }.
\]
Recalling (2) and (3) in Proposition \ref{JLScharts} we may assume that the
expression of $h$ in these coordinates is very close to $\delta_{ij}$ and that
$D^{2}\varphi$ is not large.

Differentiating the components of the induced metric gives
\begin{multline}
\partial_{x^{p}}g_{ij}=\text{function of }\left(  x,D\varphi\right)
\label{aaa}\\
+\text{Terms involving up to three factors of }D^{2}\varphi\text{ but no
higher}\nonumber\\
+\text{Terms involving up to two factors of }D^{2}\varphi\text{ and one factor
of }D^{3}\varphi.\nonumber
\end{multline}

\begin{lemma}
\label{A2B2}In a Darboux chart, using the coordinate basis (\ref{bdd}) for the
tangent space and $\left\{  Je_{l}\right\}  $ for the normal space, the
covariant derivatives of the second fundamental form and of the potential
$\varphi$ are related by
\begin{equation}
\nabla^{k-1}A=D^{k+2}\varphi+S_{k} \label{SBA}%
\end{equation}
where $S_{1}$ is a smooth controlled function depending on the chart, $h$ and
$D^{2}\varphi$, $S_{2}$ depends also on $D^{3}\varphi$ and for $k\geq3:$

\begin{enumerate}
\item Each $S_{k}$ is a sum of of multilinear forms of $D^{4}\varphi
,...,D^{k+1}\varphi$

\item The coefficients of these forms are functions of $\left(  x,D\varphi
,D^{2}\varphi,D^{3}\varphi\right)  $

\item The total sum of the derivatives of $D^{3}\varphi$ that occur in a given
term is no more than $k-2.$
\end{enumerate}

(Note that (\ref{SBA}) is interpreted as literal equality of the symbols in
the choice of basis, not simply \textquotedblleft up to a smooth function")
\end{lemma}

\begin{proof}
Starting with $k=1$, differentiate in the ambient space
\[
\tilde{\nabla}_{e_{i}}e_{j}=\tilde{\Gamma}_{ji}^{\beta}E_{\beta}+\varphi
_{jmi}\delta^{mk}E_{n+k}+\varphi_{jm}\delta^{mk}\tilde{\Gamma}_{n+k,i}^{\beta
}E_{\beta}+\varphi_{jm}\varphi_{ri}\delta^{mk}\delta^{rl}\tilde{\Gamma
}_{n+k,n+l}^{\beta}E_{\beta}.
\]
Using $e_{j}$ and $Je_{k}$ as frame and normal frame,
\begin{align*}
A_{ijl}  &  =\langle\tilde{\nabla}_{e_{i}}e_{j},Je_{l}\rangle\\
&  =\omega(e_{l},\varphi_{jmi}\delta^{mk}E_{n+k})+\omega(e_{l},\tilde{\Gamma
}_{ji}^{\beta}E_{\beta}+\varphi_{jm}\delta^{mk}\tilde{\Gamma}_{n+k,i}^{\beta
}E_{\beta}+\varphi_{jm}\varphi_{ri}\delta^{mk}\delta^{rl}\tilde{\Gamma
}_{n+k,n+l}^{\beta}E_{\beta})\\
&  =\varphi_{jli}+S_{1}%
\end{align*}
where $S_{1}$ is a smooth function involving $D^{2}\varphi$ and the ambient
Christoffel symbols at $\left(  x,D\varphi\right)  $ and recalling
\[
\omega(e_{l},E_{\beta})=\omega(E_{l}+\varphi_{lk}\delta^{jk}E_{j+n},E_{\beta
})=\left\{
\begin{array}
[c]{cl}%
\delta_{sl}, & \ \ \ \text{ if }\beta=s+n\text{ for }s\in\left\{
1,...,n\right\} \\
-\varphi_{sl}, & \,\,\,\,\text{ if }\beta=s\in\left\{  1,...,n\right\}  .
\end{array}
\right.
\]
Now for $k=2$ ($\nabla$ denotes covariant derivatives on the submanifold):
\[
\left(  \nabla A\right)  _{pijl}=\partial_{p}A_{ijl}-A\left(  \nabla_{e_{p}
}e_{i},e_{j},e_{l}\right)  -A\left(  e_{i},\nabla_{e_{p}}e_{j},e_{l}\right)
-A\left(  e_{i},e_{j},\nabla_{e_{p}}e_{l}\right)  .
\]
Now we can compute the Christoffel symbols with respect to the induced metric
$g$:
\[
\nabla_{e_{p}}e_{i}=1\ast D^{3}\varphi\text{ + lower order }
\]
thus
\begin{align*}
\left(  \nabla A\right)  _{pijl}  &  =\varphi_{jlip}+\partial_{x^{p}}%
S_{1}-A\ast\left(  D^{3}\varphi\right)  +\text{lower order}\\
&  =\varphi_{jlip}+D^{3}\varphi\ast D^{3}\varphi+1\ast D^{3}\varphi+\text{
smooth in other arguments.}%
\end{align*}
Here and in sequel, we use $A\ast B$ to denote a predictable linear
combination of terms from tensors $A$ and $B$, and $1\ast T$ to be a
predictable linear combination of $T.$ \ \ Now
\begin{align*}
\nabla^{2}A  &  =D^{5}\varphi+D^{4}\varphi\ast D^{3}\varphi+\text{lower
order}\\
\nabla^{3}A  &  =D^{6}\varphi+D^{5}\varphi\ast D^{3}\varphi+D^{4}\varphi\ast
D^{4}\varphi\text{ + lower order}%
\end{align*}
and so forth.
The result follows by inductively applying the product rule and noting
\[
\nabla^{k-1}A=D(\nabla^{k-2}A)+D^{3}\varphi\ast\nabla^{k-2}A\text{ + lower
order}%
\]
by the formula for covariant derivative.

\end{proof}

\subsection{Integral inequalities}

We will use $\Vert\cdot\Vert_{\infty}$ for the supremum norm in the euclidean
metric $\delta_{0}$ and
\[
\left\vert D^{m}\varphi\right\vert _{g}^{2}=g^{i_{1}j_{1}}g^{i_{2}j_{2}%
}...g^{i_{m}j_{m}}\varphi_{i_{1}...i_{m}}\varphi_{j_{1}...j_{m}}%
\]
to denote the norm squared with respect to $g$ for the locally defined
$m$-tensor $D^{m}\varphi$ instead of the higher covariant derivative tensor
$\nabla^{m}\varphi$. We find that this makes computations on the chosen
Darboux chart more transparent. Note that since $g$ is close to $\delta_{0}$
on the chart (with estimates on errors)
\[
\frac{\left\vert D^{m}\varphi\right\vert _{g}^{2}}{\left\vert D^{m}%
\varphi\right\vert _{\delta_{0}}^{2}}\in(1-c_{n},1+c_{n})
\]
and
\[
(1-c_{n}) dx\leq dV_{g}\leq\left(  1+c_{n}\right)  dx
\]
for some small $c_{n}$. Thus we may regard as equivalent estimates on
integrals against $dx$ and integrals against $dV_{g}$, provided the quantities
we are integrating are nonnegative. However, if the quantity being integrated
is not known to be non-negative, we have to be precise in performing estimates.

All estimates below implicitly depend on $c_{n},$ but $c_{n}$ need not be
tracked closely: it need not be close to zero.

\subsubsection{Interpolation inequalities}

We use Gagliardo-Nirenberg interpolation to derive integral inequalities that
allow us to integrate multilinear combinations of higher derivatives of
$\varphi$. For simplicity of notation, we will use $C$ for uniform constants
with dependence indicated in its arguments. For our application, we give
interpolations for different range of indices.


\begin{lemma}
\label{Gagliardo-Nirenberg} Let $\xi$ be a smooth compactly supported
vector-valued function on $\mathbb{R}^{n}$.

\begin{enumerate}
\item If $j_{1}+j_{2}+j_{3}+...+j_{q}=m, $ then
\[
\int\left\vert D^{j_{1}}\xi\ast D^{j_{2}}\xi...\ast D^{j_{q}}\xi\right\vert
^{2}\leq C\left\Vert \xi\right\Vert _{\infty}^{2q-2}\int\left\vert D^{m}%
\xi\right\vert ^{2}.
\]

\item If $j_{1}+j_{2}+j_{3}+...+j_{q}+j^{\ast}=2\tilde{m}$ where $j_{q}%
<\tilde{m}$ and $j^{\ast}\geq0$, then
\[
\int\left\vert D^{j_{1}}\xi\ast D^{j_{2}}\xi...\ast D^{j_{q}}\xi\right\vert
\leq C\left\Vert \xi\right\Vert _{\infty}^{q-\left(  2-j^{\ast}/2\tilde
{m}\right)  }\left(  \frac{2\tilde{m}-j^{\ast}}{2\tilde{m}}\int\left\vert
D^{\tilde{m}}\xi\right\vert ^{2}+\frac{j^{\ast}}{2\tilde{m}}\left\Vert
\chi_{\text{supp}\left(  \xi\right)  }\right\Vert _{p^{\ast}}^{2\tilde
{m}/j^{\ast}}\right)  .
\]

\item If $j_{1}+...+j_{r}=2\bar{m}+1 $ and all $j_{i}\leq\bar{m},$ then for
$\varepsilon>0$
\[
\int\left\vert D^{j_{1}}\xi\ast D^{j_{2}}\xi...\ast D^{j_{r}}\xi\right\vert
\leq\varepsilon\int\left\vert D^{\bar{m}+1}\xi\right\vert ^{2}+C(\varepsilon
,\bar{m},\Vert\xi\Vert_{\infty})\,\left(  \int\left\vert D^{\bar{m}}%
\xi\right\vert ^{2}+\int\chi_{\text{supp}\left(  f\right)  }\right)  .
\]

\end{enumerate}
\end{lemma}


\begin{proof}
For (1), use $p_{i}=\frac{m}{j_{i}}$ and apply the generalized H\"{o}lder's
inequality
\[
\int\left\vert D^{j_{1}}\xi\ast D^{j_{2}}\xi...\ast D^{j_{q}}\xi\right\vert
^{2}\leq\left\Vert D^{j_{1}}\xi\right\Vert _{2p_{1}}^{2}...\left\Vert
D^{j_{q}}\xi\right\Vert _{2p_{q}}^{2}%
\]
and then use the Gagliardo-Nirenberg interpolation inequality (cf.
\cite[Theorem 1.1]{MR4237368}) with $\theta_{i}=\frac{j_{i}}{m}$.
\ifthenelse{\boolean{bluecomments}}{\textcolor{blue}{
For reference
\[
\left\Vert D^{j}f\right\Vert _{p}\leq C\left\Vert D^{m}f\right\Vert
_{r}^{\theta}\left\Vert f\right\Vert _{q}^{1-\theta}\]
for
\[
\frac{1}{p}=\frac{j}{n}+\theta\left(  \frac{1}{r}-\frac{m}{n}\right)
+\frac{1-\theta}{q}\]
}}{}

For (2), taking $p_{i}=\frac{2\tilde{m}}{j_{i}}$ and $p^{\ast}=\frac
{2\tilde{m}}{j^{\ast}}$\ if $j^{\ast}>0$, then\
\[
\int\left\vert D^{j_{1}}\xi\ast D^{j_{2}}\xi...\ast D^{j_{q}}\xi\right\vert
\leq\left\Vert D^{j_{1}}\xi\right\Vert _{p_{1}}...\left\Vert D^{j_{q}}%
\xi\right\Vert _{p_{q}}\left\Vert \chi_{\text{supp}\left(  \xi\right)
}\right\Vert _{p^{\ast}}.
\]
Now apply the Gagliardo-Nirenberg interpolation inequality with $\theta
_{i}=\frac{j_{i}}{\tilde{m}}$, applying Young's inequality if $j^{\ast}>0$.

For (3), we may split, with $a\leq\bar{m}\leq\bar{m}+1\leq b$
\begin{align*}
j_{1}+...+j_{s}  &  =a\\
j_{s+1}+...+j_{r}  &  =b.
\end{align*}
Now for some $p,q$ conjugates to be determined, let%
\[
p_{i}=\left\{
\begin{array}
[c]{cl}%
\frac{ap}{j_{i}}, & \ \ \ \text{ if }i\in\left\{  1,...,s\right\} \\
\frac{bq}{j_{i}}, & \,\,\,\,\text{ if }i\in\left\{  s+1,...,r\right\}  .
\end{array}
\right.
\]
\ifthenelse{\boolean{bluecomments}}{\textcolor{blue}{
Note that
\begin{align*}
\sum\frac{1}{p_{i}}  &  =\sum_{1}^{s}\frac{j_{i}}{ap}+\sum_{s+1}^{q}\frac{j_{i}}{bq} =\frac{1}{ap}a+\frac{1}{bq}b=1.
\end{align*}}}{} Apply the generalized H\"{o}lder's inequality
\begin{equation}
\int\left\vert D^{j_{1}}\xi\ast D^{j_{2}}\xi...\ast D^{j_{r}}\xi\right\vert
\leq\left\Vert D^{j_{1}}\xi\right\Vert _{p_{1}}...\left\Vert D^{j_{q}}%
\xi\right\Vert _{p_{r}}. \label{i31}%
\end{equation}
We have from the Gagliardo-Nirenberg interpolation inequality%
\[
\left\Vert D^{j_{i}}\xi\right\Vert _{p_{i}}\leq\left\{
\begin{array}
[c]{cl}%
\left\Vert D^{\bar{m}}\xi\right\Vert _{\frac{ap}{\bar{m}}}^{\frac{j_{i}}%
{\bar{m}}}\left\Vert \xi\right\Vert _{\infty}^{1-\frac{j_{i}}{\bar{m}}} &
\ \ \ \text{ if }i\in\left\{  1,...,s\right\}  \text{ with }\theta_{i}%
=\frac{j_{i}}{\bar{m}}\\
\left\Vert D^{\bar{m}+1}\xi\right\Vert _{\frac{bq}{\bar{m}+1}}^{\frac{j_{i}%
}{\bar{m}+1}}\left\Vert \xi\right\Vert _{\infty}^{1-\frac{j_{i}}{\bar{m}+1}%
}, & \,\,\,\,\text{ if }i\in\left\{  s+1,...,r\right\}  \text{ with }%
\theta_{i}=\frac{j_{i}}{\bar{m}+1}.
\end{array}
\right.
\]

\ifthenelse{\boolean{bluecomments}}{\textcolor{blue}{
Recall if we want $q=\infty$ in
\[
\frac{1}{p_{i}}=\frac{j_{i}}{ap}=\frac{j}{n}+\theta\left(  \frac{1}{r}-\frac{\bar{m}}{n}\right)  +\frac{1-\theta}{q}\]
we solve by first choosing $\theta$ so that dimension $n$ does not appear,
namely
\begin{align*}
\frac{j_{i}}{ap}  &  =\frac{j}{n}+\frac{j_{i}}{\bar{m}}\left(  \frac{1}{r}-\frac{\bar{m}}{n}\right)  +\frac{1-\theta}{q}=\frac{j}{n}-\frac{\bar{m}}{n}\frac{j_{i}}{\bar{m}}+\frac{j_{i}}{\bar{m}}\frac{1}{r}\end{align*}
leaving us with
\[
\frac{j_{i}}{ap}=\frac{j_{i}}{\bar{m}}\frac{1}{r},\,\,\, r=\frac{ap}{\bar{m}}.
\]
}}{} Taking the product and then applying Young's inequality (for the same
$p,q$)
\begin{align}
\left\Vert D^{j_{1}}\xi\right\Vert _{p_{1}}...\left\Vert D^{j_{q}}%
\xi\right\Vert _{p_{r}}  &  \leq\left\Vert D^{\bar{m}}\xi\right\Vert
_{\frac{ap}{\bar{m}}}^{\frac{a}{\bar{m}}}\left\Vert D^{\bar{m}+1}%
\xi\right\Vert _{\frac{bq}{\bar{m}+1}}^{\frac{b}{\bar{m}+1}}\left\Vert
\xi\right\Vert _{\infty}^{r-\frac{a}{m}-\frac{b}{m+1}}\nonumber\\
&  \leq C(\varepsilon,p,q,r,\left\Vert \xi\right\Vert _{\infty})\left\Vert
D^{\bar{m}}\xi\right\Vert _{\frac{ap}{\bar{m}}}^{\frac{ap}{\bar{m}}%
}+\varepsilon\left\Vert D^{\bar{m}+1}\xi\right\Vert _{\frac{bq}{\bar{m}+1}%
}^{\frac{bq}{\bar{m}+1}}\nonumber\\
&  =C(\varepsilon,p,q,r,\left\Vert \xi\right\Vert _{\infty})\left\Vert
D^{\bar{m}}\xi\right\Vert _{2\frac{a(\bar{m}+1)}{\bar{m}(a+1)}}^{2\frac
{a(\bar{m}+1)}{\bar{m}(a+1)}}+\varepsilon\left\Vert D^{\bar{m}+1}%
\xi\right\Vert _{2}^{2} \label{i32}%
\end{align}
where in the last line we have made the choices
\[
q=\frac{2(\bar{m}+1)}{b},\text{ \ }p=\frac{2(\bar{m}+1)}{a+1}.
\]
Since $1\leq a\leq\bar{m}$ we have
\[
\frac{a(\bar{m}+1)}{\bar{m}(a+1)}\leq1
\]
and can use H\"{o}lder's and Young's inequalities to get
\begin{equation}
C(\varepsilon,p,q)\left\Vert D^{\bar{m}}\xi\right\Vert _{2\frac{a(\bar{m}%
+1)}{\bar{m}(a+1)}}^{2\frac{a(\bar{m}+1)}{\bar{m}(a+1)}}\leq C(a,\bar
{m})\left(  \int\left\vert D^{\bar{m}}\xi\right\vert ^{2}+\int\chi
_{\text{supp}\left(  f\right)  }\right)  \label{i33}%
\end{equation}
omitting the last term in the case $a=\bar{m}$. Chaining together (\ref{i31},
\ref{i32}, \ref{i33}) gives the result.

\end{proof}

\begin{lemma}
\label{interpolation} Let $f\in C^{\infty}(B_{4})$ and $r_{1}<r_{2}\leq4$.

\begin{enumerate}
\item If $j_{1}+...+j_{s}=m, $ then
\[
\int_{B_{r_{1}}}\left\vert D^{j_{i}}f\cdots D^{j_{s}}f\right\vert ^{2}\leq
C(m,r\,_{1},r_{2})\,\Vert f\Vert_{\infty}^{2s-2}\sum_{j=0}^{m}\int_{B_{r_{2}}%
}\left\vert D^{j}f\right\vert ^{2}.
\]

\item If $j_{1}+j_{2}+j_{3}+...+j_{s}+j^{\ast}=2\tilde{m} $ where
$j_{q}<\tilde{m}$ and $j^{\ast}\geq0,$ then
\[
\int_{B_{r_{1}}}\left\vert D^{j_{i}}f\cdots D^{j_{s}}f\right\vert \leq
C(\tilde{m},r_{1},r_{2})\,\left\Vert f\right\Vert _{\infty}^{s-\left(
2-j^{\ast}/2\tilde{m}\right)  }\left(  \frac{2\tilde{m}-j^{\ast}}{2\tilde{m}%
}\sum_{j=0}^{\tilde{m}}\int_{B_{r_{2}}}\left\vert D^{j}f\right\vert ^{2}%
+\frac{j^{\ast}}{2\tilde{m}}\left\Vert \chi_{\text{supp}\left(  f\right)
}\right\Vert _{p^{\ast}}^{2\tilde{m}/j^{\ast}}\right)  .
\]

\item If $j_{1}+...+j_{r}=2\bar{m}+1 $ and all $j_{i}\leq\bar{m},$ then for
$\varepsilon>0$
\[
\int_{B_{r_{1}}}\left\vert D^{j_{1}}f\ast D^{j_{2}}f...\ast D^{j_{r}%
}f\right\vert \leq\varepsilon\int_{B_{r_{2}}}\left\vert D^{\bar{m}+1}%
\xi\right\vert ^{2}+C(\varepsilon,r_{1},r_{2},\bar{m},\Vert f\Vert_{\infty
})\,\left(  \sum_{j=0}^{\tilde{m}}\int_{B_{r_{2}}}\left\vert D^{j}%
\xi\right\vert ^{2}+1\right)  .
\]

\end{enumerate}
\end{lemma}

\begin{proof}
Set $\tilde{\eta}\in C_{0}^{\infty}(B_{3})$ that is 1 on $B_{r_{1}}$, 0 on
$B_{3}(0)\backslash B_{r_{2}}(0)$, $0\leq\tilde{\eta}\leq1$ and $\Vert
\tilde{\eta} \Vert_{C^{m}}\leq C(m)$. By Lemma \ref{Gagliardo-Nirenberg}
(second line below)
\begin{align*}
\int_{B_{r_{1}}}\left\vert D^{i_{1}}f\cdots D^{i_{s}}f\right\vert ^{2}  &
\leq\int_{B_{r_{2}}}\left\vert D^{i_{1}}(\tilde{\eta}f)\cdots D^{i_{s}}%
(\tilde{\eta}f)\right\vert ^{2}\\
&  \leq\Vert\tilde{\eta}f\Vert_{\infty}^{2s-2}\int_{B_{r_{2}}}\left\vert
D^{m}(\tilde{\eta}f)\right\vert ^{2}\\
&  \leq C(m,r_{1},r_{2},\Vert\tilde{\eta}\Vert_{C^{m}})\,\Vert f\Vert_{\infty
}^{2s-2}\sum_{j=0}^{m}\int_{B_{r_{2}}}\left\vert D^{j}f\right\vert ^{2}.
\end{align*}
The second and third inequalities in the statement of the Lemma follows by
applying the previous Lemma in a similar way.
\end{proof}

The following is simple but will be used repeatedly, so we explicitly note it.

\begin{lemma}
\label{used in sum}Suppose that $r_{1}<r_{2}.$ \ Then for $\tilde{\varepsilon
}>0$
\[
\int_{B_{r_{1}}}\left\vert D^{k+3}\varphi\right\vert ^{2}\leq\tilde
{\varepsilon}\int_{B_{r_{2}}}\left\vert D^{k+4}\varphi\right\vert
^{2}+C(\tilde{\varepsilon},r_{1},r_{2})\int_{B_{r_{2}}}\left\vert
D^{k+2}\varphi\right\vert ^{2}.
\]

\end{lemma}

\begin{proof}
For some $\tilde{\eta}=1$ on $B_{r_{1}}$ supported inside $B_{r_{2}}$
\begin{align*}
\int_{B_{r_{2}}}\left\vert D^{k+3}\varphi\right\vert ^{2}\tilde{\eta}^{2}  &
=-\int_{B_{r_{2}}}\ D^{k+2}\varphi\ast\left(  \tilde{\eta}^{2}D^{k+4}%
\varphi+2\tilde{\eta}D\tilde{\eta}D^{k+3}\varphi\right) \\
&  \leq\tilde{\varepsilon}\int_{B_{r_{2}}}\tilde{\eta}^{2}\left\vert
D^{k+4}\varphi\right\vert ^{2}+\frac{1}{\tilde{\varepsilon}}C\int_{B_{r_{2}}%
}\tilde{\eta}^{2}\left\vert D^{k+2}\varphi\right\vert ^{2}\\
&  +\frac{1}{2}\int_{B_{r_{2}}}\tilde{\eta}^{2}\left\vert D^{k+3}%
\varphi\right\vert ^{2}+C\int_{B_{r_{2}}}\left\vert D\tilde{\eta}\right\vert
^{2}\left\vert D^{k+2}\varphi\right\vert ^{2}.
\end{align*}
Thus
\[
\int_{B_{r_{1}}}\left\vert D^{k+3}\varphi\right\vert ^{2}\leq2\tilde
{\varepsilon}\int_{B_{r_{2}}}\tilde{\eta}^{2}\left\vert D^{k+4}\varphi
\right\vert ^{2}+\frac{2}{\tilde{\varepsilon}}C\int_{B_{r_{2}}}\left(
\tilde{\eta}^{2}+\left\vert D\tilde{\eta}\right\vert ^{2}\right)  \left\vert
D^{k+2}\varphi\right\vert ^{2}.
\]

\end{proof}

\subsubsection{Evolution inequalities\label{evo}}

\begin{proposition}
\label{elliptic} Let $\rho_{\alpha}^{2}\in C_{0}^{\infty}(B_{2}(0))$ defined
by (\ref{bump}). \ Working in Darboux charts, for $\varepsilon>0$ we have
\begin{align}
\int_{B_{2}}\frac{d}{dt}\left(  \left\vert D^{k+2}\varphi\right\vert _{g}%
^{2}dV_{g}\right)  \rho_{\alpha}^{2}  &  \leq-2\int_{B_{2}}\left\vert
D^{k+4}\varphi\right\vert _{g}^{2}\rho_{\alpha}^{2}dV_{g}+\varepsilon
\int_{B_{3}}|D^{k+4}\varphi|^{2}dV_{g}\label{ebound}\\
&  +C(k,\varepsilon,\left\Vert \varphi\right\Vert _{C^{3}})\left(  \sum
_{m=3}^{k+2}\int_{B_{3}}\left\vert D^{m}\varphi\right\vert _{g}^{2}%
dV_{g}+1\right)  .\nonumber
\end{align}

\end{proposition}

\begin{proof}
In a Darboux chart, express $dV_{g}=V_{g}dx$. We have
\begin{align}
\frac{d}{dt}\left(  \left\vert D^{k+2}\varphi\right\vert _{g}^{2}V_{g}\right)
=\,  &  2\left(  \partial_{t}\varphi_{i_{1}...i_{k+2}}\right)  \left(
\varphi_{j_{1}...j_{k+2}}g^{i_{1}j_{1}}g^{i_{2}j_{2}}...g^{i_{k+2}j_{k+2}%
}V_{g}\right) \nonumber\\
&  +\varphi_{i_{1}...i_{k+2}}\varphi_{j_{1}...j_{k+2}}\partial_{t}\left(
g^{i_{1}j_{1}}g^{i_{2}j_{2}}...g^{i_{k+2}j_{k+2}}V_{g}\right) \nonumber\\
=\,  &  -2(g^{kl}g^{pq}\varphi_{klpq}+G)_{i_{1}...i_{k+2}}\left(
\varphi_{j_{1}...j_{k+2}}g^{i_{1}j_{1}}...g^{i_{k+2}j_{k+2}}V_{g}\right)
\label{f113}\\
&  +\varphi_{i_{1}...i_{k+2}}\varphi_{j_{1}...j_{k+2}}\partial_{t}\left(
g^{i_{1}j_{1}}g^{i_{2}j_{2}}...g^{i_{k+2}j_{k+2}}V_{g}\right)  .\nonumber
\end{align}
We count the highest order of derivatives of $\varphi$ in $x^{1},...,x^{n}$
for each term below:

\begin{enumerate}
\item $g,g^{-1},V_{g}$ are of 2nd order

\item $\partial_{t}g,\partial_{t}g^{-1}$ and $\partial_{t}V_{g}=V_{g}%
g^{ij}\partial_{t}g_{ij}$ are of 6th order

\item $(g^{kl}g^{pq}\varphi_{klpq})_{i_{1}...i_{k+2}}$ is of $(k+6)$th order
and $G_{i_{1}...i_{k+2}}$ is of $(k+5)$th order.
\end{enumerate}

For the sake of notation, we will use

\begin{enumerate}
\item $P=P(x,D\varphi,D^{2}\varphi,D^{3}\varphi)$,

\item $Q=P(x,D\varphi,...,D^{k+3}\varphi)\ $to be described in (\ref{qanon})
and below.

\item Bounded second order quantities are absorbed and not explicitly stated
unless necessary (in particular, $dV_{g}$ will be dropped when not being differentiated.)
\end{enumerate}

Multiplying $\rho_{\alpha}^{2}$ to localize in a chart then integrate on $L$,
we may then perform:

\vspace{.2cm}

\noindent{(\textbf{A}}) Integration by parts twice the first term in
(\ref{f113}) leads to

%

\begin{equation}
\int_{B_{2}}(-2g^{kl}g^{pq}\varphi_{klpq})_{i_{1}...i_{k+2}}\left(
\varphi_{j_{1}...j_{k+2}}g^{i_{1}j_{1}}...g^{i_{k+2}j_{k+2}}V_{g}\right)
\rho_{\alpha}^{2}=-2\int_{B_{2}}\left\vert D^{k+4}\varphi\right\vert _{g}%
^{2}\rho_{\alpha}^{2}\,dV_{g}+I \label{Ipart}%
\end{equation}
where%
\[
I=\int_{B_{2}}D^{k+4}\varphi\ast\left(  D^{4}\varphi+P^{2})\ast D^{k+2}%
\varphi\rho_{\alpha}^{2}+D^{k+3}\varphi\ast\left(  P\ast\rho_{\alpha}%
^{2}+D\rho_{\alpha}^{2}\right)  +D^{k+2}\varphi\ast\left(  P\ast D\rho
_{\alpha}^{2}+D^{2}\rho_{\alpha}^{2}\right)  \right)  .
\]
To deal with the first term in $I$:%
\begin{equation}
\int_{B_{2}}D^{k+4}\varphi\ast\left(  D^{4}\varphi+P^{2}\right)  \ast
D^{k+2}\varphi\leq\varepsilon\int_{B_{2}}\rho_{\alpha}^{2}|D^{k+4}\varphi
|^{2}+C(\varepsilon)\int_{B_{2}}\rho_{\alpha}^{2}\left(  \left\vert
D^{k+2}\varphi\ast\left(  D^{4}\varphi+P^{2}\right)  \right\vert ^{2}\right)
. \label{I11m}%
\end{equation}
Lemma \ref{interpolation} with $m=k$, $f=$ $D^{3}\varphi$ and $r_{2}=5/2$
yields
\[
\int_{B_{2}}\left\vert D^{k+2}\varphi\ast D^{4}\varphi\right\vert ^{2}\leq
C(D^{3}\varphi)\sum_{j=0}^{m}\int_{B_{5/2}}\left\vert D^{j+3}\varphi
\right\vert ^{2}.
\]
Applying Lemma (\ref{used in sum}) to the highest order term provides a bound
of (\ref{I11m}) by the positive terms in (\ref{ebound}) noting also that%
\[
\int_{B_{2}}\left\vert D^{k+2}\varphi\ast P^{2}\right\vert ^{2}\leq
C(D^{3}\varphi)\int_{B_{2}}\left\vert D^{k+2}\varphi\right\vert ^{2}.
\]
Next
\begin{equation}
\int_{B_{2}}\left\vert (P\ast\rho_{\alpha}^{2}+D\rho_{\alpha}^{2}%
)D^{k+4}\varphi\ast D^{k+3}\varphi\right\vert \leq\varepsilon\int_{B_{2}}%
\rho_{\alpha}^{2}|D^{k+4}\varphi|^{2}+\frac{1}{\varepsilon}C(P)\int_{B_{2}%
}\left\vert D^{k+3}\varphi\right\vert ^{2} \label{I11}%
\end{equation}
recalling that $D\rho_{\alpha}$ is bounded by uniform constants and
$D^{2}\varphi.$ By Lemma \ref{used in sum} (choosing $\tilde{\varepsilon
}\approx c\varepsilon^{2}$), (\ref{I11}) is bounded by
\[
\varepsilon\int_{B_{2}}\rho_{\alpha}^{2}|D^{k+4}\varphi|^{2}+\varepsilon
\int_{B_{3}}\left\vert D^{k+4}\varphi\right\vert ^{2}+\frac{1}{\varepsilon
^{3}}C\int_{B_{3}}\left\vert D^{k+2}\varphi\right\vert ^{2}%
\]
which is of the correct form. Finally for $I$, using (\ref{rhobound}),
\[
\int_{B_{2}}\left\vert D^{k+4}\varphi\ast D^{k+2}\varphi\ast\left(  P\ast
D\rho_{\alpha}^{2}+D^{2}\rho_{\alpha}^{2}\right)  \right\vert \leq
\varepsilon\int_{B_{2}}\rho_{\alpha}^{2}|D^{k+4}\varphi|^{2}+C(\varepsilon
,P)\int_{B_{2}}\left\vert D^{k+2}\varphi\right\vert ^{2}.
\]

\vspace{.2cm}

\noindent(\textbf{B}) {Note that when applying the product rule successively
to }$G$, we will get

\begin{enumerate}
\item A single highest order term which is linear in the highest order with
coefficients involving at most order {$D^{3}\varphi.$}

\item Second to highest order terms that are linear in the second highest
order, may have a factor of {$D^{4}\varphi$, all other dependence of lower
order. }

\item Terms of lower order, which could be multilinearly dependent on various
lower orders.
\end{enumerate}

Thus
\begin{equation}
G_{i_{1}...i_{k+1}}=1\ast D^{k+4}\varphi+Q. \label{qanon}%
\end{equation}
with $Q$ having highest order $D^{k+3}\varphi.$ \ This is observed by
iterating the following expansion: \ Using $DG$ to denote a derivative in $x$
of the composition
\[
x\mapsto G(x,D\varphi(x),D^{2}\varphi(x),D^{3}\varphi(x))
\]
and $\bar{D}G$ to denote derivatives in all 4 arguments of $G,$ we have \
\[
DG=\bar{D}G\ast\left(  D^{4}\varphi+D^{3}\varphi+D^{2}\varphi+\phi\right)
\]
where $\phi$ is the term generated by $\bar{D}G/Dx.$ \ Continuing
\begin{align}
D^{2}G  &  =\bar{D}G\ast\left(  D^{5}\varphi+D^{4}\varphi+D^{3}\varphi
+D\phi\ast\left(  D^{4}\varphi+D^{3}\varphi+D^{2}\varphi\right)  \right)
\nonumber\\
&  +\bar{D}^{2}G\ast\left(  D^{4}\varphi+D^{3}\varphi+D^{2}\varphi
+\phi\right)  \ast\left(  D^{4}\varphi+D^{3}\varphi+D^{2}\varphi+\phi\right)
\nonumber\\
&  ...\nonumber\\
D^{k+1}G  &  =\bar{D}G\ast\left(  D^{k+4}\varphi+D^{k+3}\varphi+D^{k+2}%
\varphi+...\right) \label{q5}\\
&  +\bar{D}^{2}G\ast\left(  D^{k+3}\varphi+D^{k+2}\varphi+D^{k+1}%
\varphi+...\right)  \ast\left(  D^{4}\varphi+D^{3}\varphi+D^{2}\varphi
+\phi\right) \nonumber\\
&  +\bar{D}^{3}G\ast\left(  D^{k+2}\varphi+...\right)  \ast\left\{  \left(
D^{5}\varphi+...\right)  +\left(  D^{4}\varphi+...\right)  \ast\left(
D^{4}\varphi+...\right)  \right\} \nonumber\\
&  ...\nonumber
\end{align}
Now integrate by parts:%
\begin{align*}
\int_{B_{2}}  &  G_{i_{1}...i_{k+2}}\left(  \varphi_{j_{1}...j_{k+2}}%
g^{i_{1}j_{1}}...g^{i_{k+2}j_{k+2}}V_{g}\right)  \rho_{\alpha}^{2}%
=-\int_{B_{2}}G_{i_{1}...i_{k+1}}\partial_{{i_{k+2}}}\left[  \rho_{\alpha}%
^{2}\left(  \varphi_{j_{1}...j_{k+2}}g^{i_{1}j_{1}}...g^{i_{k+2}j_{k+2}}%
V_{g}\right)  \right] \\
=  &  \int_{B_{2}}\left(  D^{k+4}\varphi\ast D^{k+3}\varphi\right)
\rho_{\alpha}^{2}+\int_{B_{2}}(D\rho_{\alpha}^{2}+\rho_{\alpha}^{2}%
P)D^{k+4}\varphi\ast D^{k+2}\varphi\\
&  +\int_{B_{2}}\left(  Q\ast D^{k+3}\varphi\right)  \rho_{\alpha}^{2}%
+\int_{B_{2}}(D\rho_{\alpha}^{2}+\rho_{\alpha}^{2}P)Q\ast D^{k+2}\varphi.
\end{align*}
We use Peter-Paul's inequality we split into two types of terms:
\[
\varepsilon\int_{B_{2}}\left\vert D^{k+4}\varphi\right\vert ^{2}\,\rho
_{\alpha}^{2}+\varepsilon\int_{B_{2}}Q^{2}\rho_{\alpha}^{2}%
\]
and
\[
C(\varepsilon,D\rho_{\alpha})\int_{B_{2}}\left(  \left\vert D^{k+3}%
\varphi\right\vert ^{2}+\left\vert D^{k+2}\varphi\right\vert ^{2}\right)
(P^{2}+P+1)\rho_{\alpha}^{2}.
\]
First,
\[
\int_{B_{2}}\left\vert D^{k+3}\varphi\right\vert ^{2}\left(  P^{2}+P+1\right)
\rho_{\alpha}^{2}\leq C\left(  D^{3}\varphi\right)  \int_{B_{2}}\left\vert
D^{k+3}\varphi\right\vert ^{2}%
\]
is bounded by the argument in Lemma (\ref{used in sum}).

One can prove by induction, observing (\ref{qanon}) and (\ref{q5}), that for
each term in $Q,$ the total number of derivatives of $D^{3}\varphi$ that arise
will sum up to no more than $k+1$ (i.e. $D^{k-1}\varphi\ast D^{6}\varphi\ast
D^{5}\varphi=D^{3+k-4}\varphi\ast D^{3+3}\varphi\ast D^{3+2}\varphi$, $\ $here
$k-4+3+2=k+1.$) \ {Applying Lemma \ref{interpolation} for $f=D^{3}\varphi$
}$\ $and $m=k+1$ to each of the squared terms gives%
\[
\int_{B_{2}}Q^{2}dV_{g}\leq C\left(  \Vert D^{3}\varphi\Vert\right)
\sum_{i=0}^{k+1}\int_{B_{3}}|D^{i+3}\varphi|^{2}%
\]
thus $\varepsilon\int_{B_{2}}Q^{2}dV_{g}$ has the correct bound, by Lemma
\ref{used in sum}.

Finally we finish bounding the last term in (\ref{f113})
\[
\int\varphi_{i_{1}...i_{k+2}}\varphi_{j_{1}...j_{k+2}}\partial_{t}\left(
g^{i_{1}j_{1}}g^{i_{2}j_{2}}...g^{i_{k+2}j_{k+2}}V_{g}\right)  \rho_{\alpha
}^{2}=\int\left(  D^{k+2}\varphi\ast D^{k+2}\varphi\ast D^{6}\varphi\right)
\rho_{\alpha}^{2}.
\]
Apply Lemma \ref{interpolation} for $f=D^{3}\varphi$ and $\tilde{m}=k$
\[
\int_{B_{2}}\left(  D^{k+2}\varphi\ast D^{k+2}\varphi\ast D^{6}\varphi\right)
\rho_{\alpha}^{2}\leq\varepsilon\int_{B_{3}}\left\vert D^{k+4}\varphi
\right\vert ^{2}+C\left(  \varepsilon,k,D^{3}\varphi\right)  \sum_{j=3}%
^{k+3}\int_{B_{3}}\left\vert D^{j}\varphi\right\vert ^{2}%
\]
{ We may }then sweep away the $\int_{B_{2}}\left\vert D^{k+3}\varphi
\right\vert ^{2}$ term Lemma \ref{used in sum} (choosing $\tilde{\varepsilon
}\approx c\varepsilon^{2}$) to conclude the proof.
\end{proof}

\begin{proposition}
\label{otherterms} Let $\rho_{\alpha}^{2}\in C_{0}^{\infty}(B_{2}(0))$.
Considering the decomposition in Lemma \ref{A2B2}, for $\varepsilon>0$ we
have
\begin{align*}
\int_{B_{2}}\frac{d}{dt}\left(  \left(  \left\vert S_{k}\right\vert _{g}%
^{2}+2\langle D^{k+2}\varphi,S_{k}\rangle_{g}\right)  dV_{g}\right)
\rho_{\alpha}^{2}  &  \leq C(k,\varepsilon,D^{3}\varphi)\,\left(  \sum
_{j=3}^{k+2}\int_{B_{3}}\left\vert D^{j}\varphi\right\vert ^{2}+1\right)  & \\
&  +\varepsilon\int_{B_{3}}\left\vert D^{k+4}\varphi\right\vert ^{2}. &
\end{align*}

\end{proposition}

\begin{proof}
Recall that%
\begin{equation}
S_{k}=\left(  1+D^{3}\varphi\right)  \ast(D^{k+1}\varphi+D^{k}\varphi\ast
D^{4}\varphi+D^{k-1}\varphi\ast D^{4}\varphi\ast D^{4}\varphi+D^{k-1}%
\varphi\ast D^{5}\varphi+ ...) \label{Bexpr}%
\end{equation}
Differentiating with respect to $t$ generates product rule expansions with 4
orders of derivatives added to a factor in each term, that is (modulo lower
order geometrically controlled values like $V_{g})$
\begin{align}
\frac{d}{dt}\left\vert S_{k}\right\vert _{g}^{2}dV_{g}  &  =\left(
1+D^{3}\varphi\right)  \ast\left(  D^{k+5}\varphi+D^{k+4}\varphi\ast
D^{4}\varphi+D^{k}\varphi\ast D^{8}\varphi+...\right)  \ast S_{k}%
\label{bkexpand}\\
&  \left(  D^{6}\varphi+D^{7}\varphi\right)  \ast\left(  D^{k+1}\varphi
+D^{k}\varphi\ast D^{4}\varphi+D^{k-1}\varphi\ast D^{4}\varphi\ast
D^{4}\varphi+...\right)  \ast S_{k}.\nonumber
\end{align}

Integrate by parts:
\begin{align*}
\int_{B_{2}}  &  \left(  D^{k+5}\varphi\ast\left(  1+D^{3}\varphi\right)  \ast
S_{k}\right)  \rho_{\alpha}^{2} =\int_{B_{2}}D^{k+4}\varphi\ast\left(
D^{4}\varphi\ast S_{k}+\left(  1+D^{3}\varphi\right)  \ast DS_{k}\right)
\rho_{\alpha}^{2}\\
&  +\int_{B_{2}}\left(  D^{k+4}\varphi\ast\left(  1+D^{3}\varphi\right)  \ast
S_{k}\right)  \ast D\rho_{\alpha}^{2}\\
&  \leq\varepsilon\int_{B_{2}}\left\vert D^{k+4}\varphi\right\vert ^{2}%
\rho_{\alpha}^{2}+C(\varepsilon)\int_{B_{2}}\left(  \left\vert D^{4}%
\varphi\ast S_{k}\right\vert ^{2}+\left\vert D^{3}\varphi\ast DS_{k}%
\right\vert ^{2}\right)  \rho_{\alpha}^{2}\\
&  +C(\varepsilon)\int_{B_{2}}\left\vert \left(  1+D^{3}\varphi\right)  \ast
S_{k}\right\vert ^{2}\left\vert D\rho_{\alpha}\right\vert ^{2}.
\end{align*}
Now apply Lemma \ref{interpolation} with $f=D^{3}\varphi$ and $m=k-1$
\begin{equation}
\int_{B_{2}}\left(  \left\vert D^{4}\varphi\ast S_{k}\right\vert
^{2}+\left\vert D^{3}\varphi\ast DS_{k}\right\vert ^{2}\right)  \rho_{\alpha
}^{2}\leq C(k,D^{3}\varphi)\,\sum_{j=3}^{k+2}\int_{B_{3}}\left\vert
D^{j}\varphi\right\vert ^{2}. \label{b87}%
\end{equation}
Similarly using $\tilde{\eta}$ as in the proof of Lemma \ref{interpolation},
we have
\[
\int_{B_{2}}\left\vert \left(  1+D^{3}\varphi\right)  \ast S_{k}\right\vert
^{2}\left\vert D\rho_{\alpha}\right\vert ^{2}\leq C\left(  D\rho_{\alpha
},D^{3}\varphi\right)  \sum_{j=3}^{k+1}\int_{B_{3}}\left\vert D^{j}%
\varphi\right\vert ^{2}.
\]
Continuing with the terms in (\ref{bkexpand})
\[
\int_{B_{2}}\left(  1+D^{3}\varphi\right)  \ast\left(  D^{k+4}\varphi\ast
D^{4}\varphi\ast S_{k}\right)  \rho_{\alpha}^{2}\leq\varepsilon\int_{B_{2}%
}\left\vert D^{k+4}\varphi\right\vert ^{2}\rho_{\alpha}^{2}+C(\varepsilon
)\int_{B_{2}}\left\vert D^{3}\varphi\ast D^{4}\varphi\ast S_{k}\right\vert
^{2}\rho_{\alpha}^{2}%
\]
with the latter term enjoying the same bound as (\ref{b87}). The remaining
terms are of the form
\[
\int\left(  D^{3+j_{1}}\varphi\ast D^{3+j_{2}}\varphi\ast..\ast D^{3+j_{q}%
}\varphi\right)  \rho_{\alpha}^{2}%
\]
with $j_{1}+...+j_{q}\leq2k$ so
\[
\int_{B_{2}}\left(  D^{3+j_{1}}\varphi\ast D^{3+j_{2}}\varphi\ast..\ast
D^{3+j_{q}}\varphi\right)  \rho_{\alpha}^{2}\leq C(m,D^{3}\varphi)\,\left(
\sum_{j=3}^{k+3}\int_{B_{3}}\left\vert D^{j}\varphi\right\vert ^{2}+1\right)
\]
by Lemma \ref{interpolation} again. Applying Lemma \ref{used in sum} to
$\int_{B_{3}}\left\vert D^{k+3}\varphi\right\vert ^{2}$ completes the desired
bound for the integral of the (\ref{bkexpand}) terms.

Next
\begin{equation}
\int_{B_{2}}\frac{d}{dt}\left(  2\langle D^{k+2}\varphi,S_{k}\rangle_{g}%
dV_{g}\right)  \rho_{\alpha}^{2}dV_{g}=\int_{B_{2}}\left(  D^{k+6}\varphi\ast
S_{k}\right)  \rho_{\alpha}^{2}dV_{g}+\int_{B_{2}}D^{k+2}\varphi\ast\frac
{d}{dt}\left(  S_{k}\ast V\right)  \rho_{\alpha}^{2}. \label{exdf}%
\end{equation}
Integrating the first term by parts twice yields
\begin{align*}
\int_{B_{2}}  &  D^{k+6}\varphi\ast S_{k}\rho_{\alpha}^{2} =\int_{B_{2}%
}D^{k+4}\varphi\ast D^{2}\left(  S_{k}\ast V\right)  \rho_{\alpha}^{2}%
+D\rho_{\alpha}^{2}\ast D\left(  S_{k}\ast V\right)  +D^{2}\rho_{\alpha}%
^{2}\ast\left(  S_{k}\ast V\right) \\
&  \leq\varepsilon\int_{B_{2}}\left\vert D^{k+4}\varphi\right\vert ^{2}%
\rho_{\alpha}^{2}+C(\varepsilon)\int_{B_{2}}\left\vert D^{2}S_{k}\right\vert
^{2}\rho_{\alpha}^{2}+C\left(  \varepsilon,D^{2}\rho_{\alpha}^{2}\right)
\int_{B_{2}}\left(  \left\vert DS_{k}\right\vert ^{2}+\left\vert
S_{k}\right\vert ^{2}\right)  .
\end{align*}
Again Lemma \ref{interpolation} with $f=D^{3}\varphi$ and $\tilde{m}=k$ and
$r_{2}=5/2$
\begin{align}
\int_{B_{2}}\left\vert D^{2}S_{k}\right\vert ^{2}\rho_{\alpha}^{2}  &  \leq
C(k,D^{3}\varphi)\,\sum_{j=3}^{k+3}\int_{B_{5/2}}\left\vert D^{j}%
\varphi\right\vert ^{2}.\\
&  \leq\varepsilon\int_{B_{3}}\left\vert D^{k+4}\varphi\right\vert
^{2}+C\left(  \varepsilon,k,D^{3}\varphi\right)  \sum_{j=3}^{k+2}\int_{B_{3}%
}\left\vert D^{j}\varphi\right\vert ^{2}%
\end{align}
using Lemma \ref{used in sum}.

Now look at second term in (\ref{exdf}). Note that
\[
D^{k+2}\varphi\ast\frac{d}{dt}\left(  S_{k}\ast V\right)  =D^{k+2}\varphi\ast
D^{k+5}\varphi\ast D^{3}\varphi+D^{k+2}\varphi\ast\left(  D^{k+4}\varphi\ast
D^{4}\varphi+...\right)  .
\]
The highest order term can be dealt with via integration by parts away from
$D^{k+5}\varphi$ and then an iterated Peter-Paul, carefully choosing smaller
$\varepsilon$ and using Lemma \ref{interpolation}. For the remaining terms, we
need the third statement in Lemma \ref{interpolation} which gives
\begin{align*}
\int_{B_{2}}\left(  D^{3+j_{1}}\varphi\ast D^{3+j_{2}}\varphi\ast..\ast
D^{3+j_{q}}\varphi\right)  \rho_{\alpha}^{2}  &  \leq\varepsilon\int_{B_{5/2}
}\left\vert D^{k+4}\varphi\right\vert ^{2}\\
&  +C(\varepsilon,k,\frac{5}{2},\Vert f\Vert_{\infty})\,\left(  \sum_{j=0}%
^{k}\int_{B_{5/2}}\left\vert D^{3+j}\varphi\right\vert ^{2}+1\right)
\end{align*}
as $j_{1}+...+j_{q}=2k+1$. A final application of Lemma \ref{used in sum} to
$\int_{B_{5/2}}\left\vert D^{k+3}\varphi\right\vert ^{2}$ completes the proof.
\end{proof}

In Proposition \ref{elliptic} we have isolated `good' terms $-2\int_{B_{2}%
}\left\vert D^{k+4}\varphi\right\vert ^{2}\rho_{\alpha}^{2}. $ We would like
to use them to offset the `bad' terms of the form $\varepsilon\int_{B_{3}%
}\left\vert D^{k+4}\varphi\right\vert ^{2}dV_{g}$ that occur in Propositions
\ref{elliptic} and \ref{otherterms}. Because the expressions for
$D^{k+4}\varphi$ are different in each chart in the cover, the difficulty
arises that we cannot directly beat the terms occurring on a larger ball by
terms on a smaller ball of different charts, even when the smaller balls cover
the larger ball. To make an argument that the bad terms in a larger ball of
one chart are offset by the good terms in a smaller ball in a different chart
requires bounding the bad terms by a global, well-defined geometric quantity
involving derivatives of the second fundamental form, modulo a lower order
difference. This is the point of the following lemma.

\begin{lemma}
\label{enoughalready}Take a finite cover of charts $\Upsilon^{\alpha}$, each
over $B_{4}(0)$ and partition of unity $\rho_{\alpha}^{2}$ in $B_{2}(0)$ in
each respective chart as described by (\ref{bump}). Then
\[
\sum_{\alpha}\int_{B_{3}}\left\vert D^{k+4}\varphi\right\vert ^{2}dV_{g}%
\leq2N\sum_{m=0}^{k+1}\int_{L}\left\vert \nabla^{k+1-m}A\right\vert ^{2}%
dV_{g}+C.
\]

\end{lemma}

\begin{proof}
Note that from Lemma \ref{A2B2}
\begin{equation}
\left\vert D^{k+4}\varphi\right\vert _{g}^{2}\leq2\left\vert \nabla
^{k+1}A\right\vert _{g}^{2}+2\left\vert S_{k+2}\right\vert _{g}^{2}.\nonumber
\end{equation}
Let $\iota=\frac{1}{k+2}.$ Then we have, taking $\tilde{\eta}=1$ on each
$B_{3}(0),$ with $\tilde{\eta}\in C_{c}^{\infty}(B_{3+\iota}(0))$
\begin{align*}
\int_{B_{3}}\left\vert D^{k+4}\varphi\right\vert ^{2}dV_{g}  &  \leq
2\int_{B_{3}}\left(  \left\vert \nabla^{k+1}A\right\vert ^{2}+\left\vert
S_{k+2}\right\vert _{g}^{2}\right)  dV_{g}\\
&  \leq2\int_{B_{3}}\left\vert \nabla^{k+1}A\right\vert ^{2}dV_{g}%
+2\int_{B_{3+\iota}}\left\vert S_{k+2}\right\vert _{g}^{2}\tilde{\eta}%
^{2}dV_{g}\\
&  \leq2\int_{B_{3}}\left\vert \nabla^{k+1}A\right\vert ^{2}dV_{g}+C\left(
\sum_{m=3}^{k+3}\int_{B_{3+\iota}}\left\vert D^{m}\varphi\right\vert
^{2}dV_{g}+1\right)
\end{align*}
by Lemma \ref{interpolation}. Iterating this argument, using%
\[
\int_{B_{3+\iota}}\left\vert D^{k+3}\varphi\right\vert ^{2}dV_{g}\leq
2\int_{B_{3+\iota}}\left\vert \nabla^{k}A\right\vert ^{2}dV_{g}+C\left(
\sum_{m=3}^{k+2}\int_{B_{3+2\iota}}\left\vert D^{m}\varphi\right\vert
^{2}dV_{g}+1\right)
\]
and so forth, for a total of $k+1$ steps, we have by using $\left\vert
D^{3}\varphi\right\vert \leq$ $\left\vert A\right\vert +C$ that
\[
\int_{B_{3}}\left\vert D^{k+4}\varphi\right\vert ^{2}dV_{g}\leq2\sum
_{m=0}^{k+1}\int_{B_{3+\frac{k+1}{k+2}}}\left\vert \nabla^{k+1-m}A\right\vert
^{2}\tilde{\eta}^{2}dV_{g}+C.
\]

Now for any set of functions $\tilde{\eta}_{\alpha}$ who are $1$ on
$B_{r}\subset B_{4}$ on each chart $\Upsilon^{\alpha}$, we can bound
\begin{align*}
\sum_{\alpha}\int_{B_{4}}\left\vert \nabla^{m}A\right\vert ^{2}\tilde{\eta
}_{\alpha}^{2}dV_{g}  &  \leq\max_{x\in L}\left(  \sum_{\alpha}\tilde{\eta
}_{\alpha}^{2}(x)\right)  \int_{L}\left\vert \nabla^{m}A\right\vert ^{2}dV_{g}
\leq N\int_{L}\left\vert \nabla^{m}A\right\vert ^{2}dV_{g}.
\end{align*}
It follows that
\[
\sum_{\alpha}\int_{B_{3}}\left\vert D^{k+4}\varphi\right\vert ^{2}dV_{g}%
\leq2N\sum_{m=0}^{k+1}\int_{L}\left\vert \nabla^{k+1-m}A\right\vert ^{2}%
dV_{g}+C.
\]

\end{proof}


\subsection{Proof of the main theorem}

\begin{proof}
[Proof of Proposition \ref{cp11}]At a fixed time $t_{0}$ we may take the
ambient charts $\left\{  \Upsilon^{\alpha}\right\}  $ for a tubular
neighborhood of $L$ and subordinate partition of unity $\left\{  \rho_{\alpha
}^{2}\right\}  $ which restrict to charts (via the $x$ coordinate) for $L$
with the same partition of unity.

Differentiate
\begin{align*}
\frac{d}{dt}\int_{L}\left\vert \nabla^{k-1}A\right\vert _{g}^{2}dV_{g}  &
=\int_{L}\frac{d}{dt}\left(  \left\vert \nabla^{k-1}A\right\vert _{g}%
^{2}dV_{g}\right) \\
&  =\int_{B_{2}}\left(  \sum_{\alpha}\rho_{\alpha}^{2}\right)  \frac{d}%
{dt}\left[  \left(  \left\vert D^{k+2}\varphi\right\vert _{g}^{2}+\left\vert
S_{k}\right\vert _{g}^{2}+2\langle D^{k+2}\varphi,S_{k}\rangle_{g}\right)
dV_{g}\right] \\
&  =\sum_{\alpha}\int_{B_{2}}\frac{d}{dt}\left(  \left\vert D^{k+2}%
\varphi\right\vert _{g}^{2}dV_{g}\right)  \rho_{\alpha}^{2}\\
&  +\sum_{\alpha}\int_{B_{2}}\frac{d}{dt}\left[  \left(  \left\vert
S_{k}\right\vert _{g}^{2}+2\langle D^{k+2}\varphi,S_{k}\rangle_{g}\right)
dV_{g}\right]  \rho_{\alpha}^{2}.
\end{align*}
Thus
\begin{align}
\frac{d}{dt}\int_{L}\left\vert \nabla^{k-1}A\right\vert _{g}^{2}dV_{g}  &
\leq-2\sum_{\alpha}\int_{B_{2}}\left\vert D^{k+4}\varphi\right\vert _{g}%
^{2}\rho_{\alpha}^{2}dV_{g}+\varepsilon\sum_{\alpha}\int_{B_{3}}%
|D^{k+4}\varphi|^{2}dV_{g}\label{ROK}\\
&  +\sum_{\alpha}C(k,\varepsilon,\left\Vert \varphi\right\Vert _{C^{3}%
})\left(  \sum_{m=3}^{k+2}\int_{B_{3}}\left\vert D^{m}\varphi\right\vert
_{g}^{2}dV_{g}+1\right) \nonumber
\end{align}
by Propositions \ref{elliptic} and \ref{otherterms}. \ Now apply Lemma
\ref{enoughalready}%
\begin{align*}
\sum_{\alpha}\int_{B_{3}}|D^{k+4}\varphi|^{2}dV_{g}  &  \leq\left(
NC\sum_{m=0}^{k+1}\int_{L}|\nabla^{m}A|^{2}dV_{g}+C\right) \\
&  =NC\sum_{m=0}^{k+1}\int_{L}|\nabla^{m}A|^{2}\left(  \sum_{\alpha}%
\rho_{\alpha}^{2}\right)  dV_{g}+C\\
&  =NC\sum_{m=0}^{k+1}\sum_{\alpha}\int_{B_{2}}|\nabla^{m}A|^{2}\rho_{\alpha
}^{2}dV_{g}\\
&  \leq NC\sum_{m=0}^{k+1}\sum_{\alpha}\int_{B_{2}}2\left(  \left\vert
D^{m+3}\varphi\right\vert ^{2}+\left\vert S_{m+1}\right\vert _{g}^{2}\right)
\rho_{\alpha}^{2}dV_{g}\\
&  =2NC\sum_{\alpha}\int_{B_{2}}\left(  \left\vert D^{k+4}\varphi\right\vert
^{2}+\left\vert S_{k+2}\right\vert _{g}^{2}\right)  \rho_{\alpha}^{2}dV_{g}\\
&  +2NC\sum_{m=0}^{k}\sum_{\alpha}\int_{B_{2}}2\left(  \left\vert
D^{m+3}\varphi\right\vert ^{2}+\left\vert S_{m+1}\right\vert _{g}^{2}\right)
\rho_{\alpha}^{2}dV_{g}.
\end{align*}
Note that from Lemma \ref{used in sum}
\[
2\int_{B_{2}}\left\vert D^{k+3}\varphi\right\vert ^{2}\rho_{\alpha}^{2}%
dV_{g}\leq\frac{1}{4NC}\int_{B_{3}}|D^{k+4}\varphi|^{2}dV_{g}+C\left(
N,\left\vert D\rho_{\alpha}^{2}\right\vert \right)  \int_{B_{3}}%
|D^{k+2}\varphi|^{2}dV_{g}.
\]
Note also Lemma \ref{interpolation}, recalling (\ref{Bexpr}), then Lemma
\ref{used in sum} on the highest order resulting term gives
\begin{align*}
2\int_{B_{2}}\left\vert S_{k+2}\right\vert _{g}^{2}\rho_{\alpha}^{2}dV_{g}  &
\leq C\sum_{m=3}^{k+3}\int_{B_{5/2}}\left\vert D^{m}\varphi\right\vert
_{g}^{2}+C\\
&  \leq\frac{1}{4NC}\int_{B_{3}}|D^{k+4}\varphi|^{2}dV_{g}+C\left(
N,\left\vert D\rho_{\alpha}^{2}\right\vert \right)  \sum_{m=3}^{k+2}%
\int_{B_{3}}\left\vert D^{m}\varphi\right\vert _{g}^{2}+C.
\end{align*}
Thus
\begin{align}
\sum_{\alpha}\int_{B_{3}}|D^{k+4}\varphi|^{2}dV_{g}  &  \leq4NC\sum_{\alpha
}\int_{B_{2}}\left\vert D^{k+4}\varphi\right\vert ^{2}\rho_{\alpha}^{2}dV_{g}
+8NC\sum_{m=3}^{k+2}\sum_{\alpha}\int_{B_{2}}\left\vert D^{m}\varphi
\right\vert ^{2}\rho_{\alpha}^{2}dV_{g}\nonumber\label{ROK2}\\
&  +8NC\sum_{m=1}^{k+1}\sum_{\alpha}\int_{B_{2}}\left\vert S_{m}\right\vert
_{g}^{2}\rho_{\alpha}^{2}dV_{g}.
\end{align}
Choosing $\varepsilon<(2NC)^{-1}$ in (\ref{ROK}) in light of (\ref{ROK2}) we
have%
\[
\frac{d}{dt}\int_{L}\left\vert \nabla^{k-1}A\right\vert _{g}^{2}dV_{g}\leq
C(N,k,\left\Vert \varphi\right\Vert _{C^{3}})\left(  \sum_{m=3}^{k+2}%
\int_{B_{3}}\left\vert D^{m}\varphi\right\vert _{g}^{2}+\sum_{m=1}^{k+1}%
\int_{B_{3}}\left\vert S_{m}\right\vert _{g}^{2}+1\right)  .
\]
Applying Lemma \ref{interpolation} to the $\int\left\vert S_{m}\right\vert
_{g}^{2}$ terms and then Lemma \ref{enoughalready} to the $\left\vert
D^{m}\varphi\right\vert _{g}^{2}$ \ terms yields the result.
\end{proof}


\begin{proof}
[Proof of Theorem \ref{c1}]Suppose now that $F$ is a solution to (\ref{flow0})
with $\left\vert A\right\vert \leq K$ on $[0,T)$. Starting with
\[
\int_{L}\left\vert A\right\vert ^{2}dV_{g}(t)\leq K\operatorname{Vol}\left(
L\right)  \leq C
\]
we may apply Proposition \ref{cp11} and apply differential inequalities:
continuing with
\[
\frac{d}{dt}\int_{L}\left\vert \nabla A\right\vert ^{2}dV_{g}(t)\leq C\int%
_{L}\left\vert \nabla A\right\vert ^{2}dV_{g}(t)+C\int_{L}\left\vert
A\right\vert ^{2}dV_{g}(t)
\]
and so forth, obtaining bounds of the form
\begin{equation}
\int_{L}\left\vert \nabla^{k-1}A\right\vert ^{2}dV_{g}(t)\leq C(k,K,F_{0},T)
\label{wk2a}%
\end{equation}
for arbitrary $k$.

Now at any $t_{0}\in\lbrack0,T)$ we may take a cover $\Upsilon^{\alpha}$ as
described in Proposition \ref{JLScharts2}. By Lemma \ref{enoughalready} and
(\ref{wk2a}) we have%
\begin{equation}
\left\Vert D^{k}\varphi\right\Vert _{L^{2}(B_{3})}\leq C(k,K,F_{0},T)
\label{univE}%
\end{equation}
for all $k,$ in every chart. By Sobolev embedding theorems, we have H\"{o}lder
bounds on $D^{k}\varphi$ over $B_{2}$ for each chart. In particular, there
will be uniform bounds on $\frac{d}{dt}D\varphi$ and $\frac{d}{dt}D^{2}%
\varphi$ which control the speed of the flow in the chart and the rate of
change of the slope the manifold $L_{t}$ makes with respect to the tangent
plane at the origin in the chart. We conclude then the manifolds $L_{t}$ will
continue to be described by the set of charts taken at $t_{0}$ for
$t<\max\left\{  T,t_{0}+\tau\right\}  $ for some positive $\tau$ with an
apriori lower bound. (Perhaps we take $c_{n}$ slightly larger in
(\ref{tangentspaceclose})). By choosing $t_{0}$ near $T$ we are assured that
these fixed charts describe the flow for all values $t\in\lbrack t_{0},T)$.

Now observe that with fixed speed bounds, the paths $x\mapsto F(x,t)$ of the
normal flow are Lipschitz and hence the normal flow extends to a well-defined
continuous map
\begin{equation}
F:L\times\lbrack0,T]\rightarrow M. \label{contE}%
\end{equation}
We claim that $F\left(  \cdot,T\right)  $ is a smooth immersion. While within
a chart, the vertical maps
\begin{equation}
\bar{F}(x):=(x,d\varphi(x,t)) \label{mapE}%
\end{equation}
converge in every H\"{o}lder norm to a smooth map at $T,$ we still must argue
that the charts given by the $x$ coordinates do not collapse as $t\rightarrow
T$. This can be argued locally, using coordinates on $L_{t_{0}}$. For any
given $x\in L_{t_{0}}$ we may choose a chart such that $x\in B_{1}(0)\subset
B_{3}(0).$ We are already assuming $F$ is an immersion at $t_{0}$ so this
coordinate chart gives us a coordinate chart for the abstract smooth manifold
$L.$ For $t>t_{0}$ the normal flow $F$ is given by
\begin{equation}
F(x,t)=(\chi_{t}(x),d\varphi(\chi_{t}(x),t)) \label{smoothE}%
\end{equation}
for some local diffeomorphism $\chi_{t}(x):B_{1}(0)\rightarrow B_{2}(0)$ from
Claim \ref{moddif}, provided that $t_{0}$ is chosen close enough to $T$ such
that
\[
{\chi}_{t}(x)\in B_{2}(0)\text{ for all }x\in B_{1}(0)\text{ and }t\in\lbrack
t_{0},T).\text{ }%
\]
This choice of $t_{0}$ is possible given that $\chi_{t}(x)$ is controlled by
the normal projection of $\frac{d\bar{F}}{dt}$ and the inverse $( d\bar
{F})^{-1}$, for $\bar{F}$ defined by (\ref{mapE}), both of which are
universally controlled given (\ref{tangentspaceclose}) and (\ref{univE}).

Now because (\ref{mapE}) is uniformly smooth, it can be extended smoothly to
$[t_{0},T+\delta),$ as well as the normal flow associated to this extension.
Applying Claim \ref{moddif} (note that we may extend the flow outside $B_{3}$
in a nice way which doesn't affect the behavior in $B_{2}(0))$ we get a smooth
diffeomorphism $\chi_{T}.$ \ For $x\in B_{1}(0)$ we can compute the normal
flow $F$,
\[
F(x,T)=(\chi_{T}(x),d\varphi(\chi_{T}(x),T))
\]
which is a smooth extension of (\ref{smoothE}) to $T$, by the uniform
estimates on $\varphi$. Now $F(x,T)$ is a smooth immersion from $B_{1}(0)$
because $\chi_{T}$ is a diffeomorphism. As $x$ was chosen arbitrarily, we
conclude the continuous extension of $F$ defined in (\ref{contE}) must be a
smooth immersion from $L$ at $T$.

We may now restart the flow by Proposition \ref{stexist} with initial
immersion $F(x,T)$. The time derivatives of the new flow and $F$ agree to any
order at $T$. Therefore the new flow is a smooth extension of $F$ to
$[0,T+\varepsilon)$ for some $\varepsilon>0$. Moreover, Theorem \ref{thm31}
asserts that this is the only smooth extension.
\end{proof}

\section{Appendix}

\subsection{Submanifold with bounded second fundamental form $A$}

It is a known and frequently used fact that when $|A|$ is bounded then the
submanifold can be written as a graph over a controlled region in its tangent
space. We provide a proof below for any dimension and codimension.

\begin{proposition}
\label{A-bound} Let $L^{k}$ be a compact manifold embedded in a compact
Riemannian manifold $(M^{k+l},g)$. Suppose that the second fundamental form of
$L$ satisfies $|A|\leq K$ for some constant $K>0$. Then $L$ is locally a graph
of a vector-valued function over a ball $B_{r}(0)\subset T_{p}L$ in a normal
neighbourhood of $p\in L$ in $M$ and $r> C(M,g)(K+1)^{-1}$ for some constant
$C(M,g)>0$.
\end{proposition}

\proof {\bf Step 1.} Bound the injectivity radius of $L$ from below in terms
of $K$. Assume $M$ is isometrically embedded in some euclidean space. For the
embedding $F:L^{k} \overset{f}{\to} M^{k+l}\overset{\varphi}{\to}%
\mathbb{R}^{k+n}$, denote its second fundamental form by $\tilde{A}$ and note
that
\[
|\tilde{A}|\leq C(|A| +1)\leq C(K+1)
\]
where $C$ only depends on the isometric embedding $\varphi$.
Let $\gamma:\mathbb{S}^{1}\to L$ be a shortest geodesic loop based at a point
$p\in L$ which is parametrized by arc-length $s$. Suppose $\gamma
(0)=\gamma(a), \gamma^{\prime}(0)=\gamma^{\prime}(a)$. Take a hyperplane $P$
in $\mathbb{R}^{k+n}$ such that $P$ intersects $\gamma$ at a point $p$
orthogonally. There is a point $q\in\gamma$ where $\gamma$ meets $P$ again at
first time. The angle between the unit vectors $\gamma^{\prime}(p)$ and
$\gamma^{\prime}(q)$ in $\mathbb{R}^{k+n}$ is at least $\frac{\pi}{2}$.
Therefore
\[
\left|  \gamma^{\prime}(p)-\gamma(^{\prime}q)\right|  \geq\sqrt{2}.
\]
Since $F\circ\gamma:\mathbb{S}^{1}\overset{\gamma}{\to}{L}\overset{F}{\to
}\mathbb{R}^{k+l}$ factors through $L$ where $\gamma$ is a geodesic, we have
(cf. \cite{MR0164306}, \cite{MR0495450} for the notation of the second
fundamental form $\nabla d \phi$ of a mapping $\phi$ between Riemannian
manifolds),
\[
\nabla d (F\circ\gamma) = dF \circ\nabla(d\gamma) + \nabla d(F)(d\gamma
,d\gamma)=\nabla d(F)(d\gamma,d\gamma).
\]
Since the Christoffel symbols of $\mathbb{S}^{1}$ and of $\mathbb{R}^{n+k}$
are 0 we have
\[
\nabla d (F\circ\gamma) = (F\circ\gamma)^{\prime\prime}.
\]
Therefore
\[
(F\circ\gamma)^{\prime\prime}= \tilde A(F)(\gamma^{\prime},\gamma^{\prime}).
\]
Integrating along the portion of $\gamma$ from $p$ to $q$, we get
\[
\sqrt{2}\leq\left|  \gamma^{\prime}(p)-\gamma^{\prime}(q) \right|  \leq
\int^{q}_{p} \left|  (F\circ\gamma)^{\prime\prime}\right|  ds\leq C(K+1)\,a.
\]
We conclude that that the length $a$ has a lower bound $C/(K+1)$.

From the Gauss equations and $|\tilde A|<C(K+1)$, the sectional curvatures of
$L$ are bounded above by $C^{2}(K+1)^{2}$. We conclude $\mbox{inj}(L)\leq
C(K+1)^{-1}$ \cite[p.178]{PP2}.

\medskip

\textbf{Step 2.} Take a normal neighbourhood $U\subset L$ around a given point
$p\in L$ and assume $U$ is contained in a normal neighbourhood $V$ of $M$ at
$p$. We will use $C(g)$ for constants only depending on the ambient geometry
of $(M,g)$. Now, on $V$ we will use $\delta=\langle\cdot,\cdot\rangle
_{\mathbb{R}^{k+l}}$, to measure length of various geometric quantities
already defined in $(V,g)$. First,
\[
|A|_{\delta}\leq C(g) |A|_{g} \leq C(g)K.
\]

Identify $T_{p}L$ with $\mathbb{R}^{k}\times\{0\}\subset\mathbb{R}^{k+l}$. Let
$e_{1}(x),...,e_{k}(x)$ be the orthonormal frame on $U$ obtained by parallel
transporting an orthonormal frame $e_{1}(0),...,e_{k}(0)$ at $T_{p}L$ along
the \textit{unique} radial geodesic $r_{x}(s)$ in $(U,f^{\ast}g)$ from $0$ to
an arbitrary point $x\in L$, and let $e_{1+k}(0),...,e_{l+k}(0)$ be the
orthonormal frame of $(T_{p}L)^{\perp}$. Integrating along $\gamma_{x}(s)$
leads to
\begin{align*}
\left\vert \langle e_{i}(x),e_{j+l}(0)\rangle\right\vert  &  =\left\vert
\langle e_{i}(x),e_{j+l}(0)\rangle-\langle e_{i}(0),e_{j+l}(0)\rangle
\right\vert \\
&  =\left\vert \int_{0}^{|x|}\frac{d}{ds}\langle e_{i}(\gamma_{x}%
(s)),e_{j+l}(0)\rangle ds\right\vert \\
&  =\left\vert \int_{0}^{|x|}\langle e_{i}^{\prime}(s),e_{j+l}(0)\rangle
ds\right\vert \\
&  \leq\int_{0}^{|x|}\left\vert \langle\nabla_{\partial_{r}}^{g}e_{i}%
,e_{j+l}(0)\rangle\right\vert ds\\
&  =\int_{0}^{|x|}\left\vert \langle A(\partial_{r},e_{i})+\nabla
_{\partial_{r}}^{L}e_{i},e_{j+l}(0)\rangle\right\vert ds\\
&  \leq C(g)K\,|x|
\end{align*}
as $\nabla_{\partial_{r}}^{L}e_{i}=0$ on $L$. Therefore, there exists
$r_{0}=C(g)K^{-1}$ (where $C(g)$ may differ from the one above) such that for
any $x\in B_{r_{0}}(0)$ the projection of each $e_{i}(x)$ in each fixed normal
direction $e_{j+l}(0)$ is at most $c_{n}/\sqrt{l}$ and the norm of the
projection is no more than some universal constant $c_{n,l}$ that we get to
choose.
It is known that such $T_{x}L$ projects bijectively to $T_{p}L$. Therefore,
locally around any $x\in B_{r_{0}}(0)$, implicit function theorem asserts that
$U$ can be written as a graph over a ball in $T_{x}L$, hence as a graph over a
ball in $T_{p}L$ from the projection. The graphing functions over the fixed
reference plane $T_{P}L$ must coincide on the overlap of any pair of such
balls. This yields a \textit{global} graphing function $\mathcal{F}$ over
$B_{r_{0}}^{n}(p)\subset T_{p}L$. Moreover, $|D\mathcal{F}|\leq C(g,l)$
because $D\mathcal{F}$ is close to $T_{x}L$ which is close (measured in $l$)
to $T_{p}L$ via the projection. \endproof

\bibliographystyle{amsalpha}
\bibliography{hflow}

\end{document}